\documentclass{amsart}
\usepackage{amsfonts}
\usepackage{amsmath}
\usepackage{graphicx}
\usepackage{mathdots}
\usepackage{mathrsfs}
\usepackage[hidelinks]{hyperref}
\usepackage{comment}
\usepackage{xcolor}
\usepackage{stmaryrd} % serve per llbracket
\usepackage{tikz-cd}
\usepackage{enumerate}
\usepackage{cite}
\usepackage{orcidlink}
\usepackage{algorithm}
\usepackage{algpseudocode}

\usepackage[a4paper,
          bindingoffset=0.2in,
          left=1in,
          right=1in,
          top=1in,
          bottom=1in,
          footskip=.25in]{geometry}
\newcommand{\de}{\mathrm{d}}
\newcommand{\R}{\mathbb{R}}
\newcommand{\N}{\mathbb{N}}
\newcommand{\C}{\mathbb{C}}
\newcommand{\spaceV}{\mathscr{V}}
\newcommand{\spaceVr}{\mathscr{V}_{r}}
\renewcommand{\P}{\mathscr{P}}
\renewcommand{\H}{\mathscr{H}}

\newcommand{\Mass}{\mathbb{M}}
\newcommand{\forms}[2]{\mathscr{P}_{#1}\Lambda^{#2}}
\newcommand{\testforms}{\mathscr{D}_0^k}
\newcommand{\currents}{\mathscr{M}_0^k}

\newcommand{\averaging}{\mathscr{A}^k}
\newcommand{\leb}{\mathcal{L}}

\renewcommand{\L}{\mathscr{L}}

\newcommand{\markov}{C_M}

\newcommand{\mesh}{admissible integral $k$-mesh }

\DeclareMathOperator*{\Span}{span}
\DeclareMathOperator*{\support}{supp}
\DeclareMathOperator*{\Vol}{Vol}
\DeclareMathOperator*{\argmin}{argmin}
\DeclareMathOperator*{\argmax}{argmax}
\DeclareMathOperator*{\ddim}{dim}
\DeclareMathOperator*{\Card}{Card}

\DeclareMathOperator*{\Rank}{Rank}
\DeclareMathOperator*{\interior}{int}
\DeclareMathOperator*{\arcos}{arcos}
\DeclareMathOperator*{\dist}{dist}
\DeclareMathOperator*{\diam}{diam}
\DeclareMathOperator*{\extr}{Extr}
\DeclareMathOperator*{\vdm}{vdm}

\newcommand{\opnorm}[1]{\Vert #1 \Vert_{\mathrm{op}}}
\newcommand{\monomials}[1]{\mathcal{B}^{\mathrm{mon}}_{#1}}
\newcommand{\kmonomials}[2]{\mathcal{B}^{\mathrm{mon},#2}_{#1}}
\newcommand{\homonomials}[1]{\mathcal{B}^{\mathrm{hom}}_{#1}}
\newcommand{\khomonomials}[2]{\mathcal{B}^{\mathrm{hom},#2}_{#1}}
\newcommand{\haus}{\mathcal{H}}
\newcommand{\myqed}{$ $}
\newcommand{\e}{\boldsymbol{e}}

\newtheorem{theorem}{Theorem}
\newtheorem{proposition}{Proposition}
\newtheorem{lemma}{Lemma}
\newtheorem{corollary}{Corollary}
\theoremstyle{definition}
\newtheorem{definition}{Definition}

\theoremstyle{remark}
\newtheorem{remark}{Remark}
\newtheorem{example}{Example}

\title[Sampling, approximation, and interpolation of differential forms]{Sampling, approximation, and interpolation of differential forms by admissible integral $k$-meshes}

\author{Ludovico Bruni Bruno}
\address{Dipartimento di Matematica \textquotedblleft Tullio Levi-Civita\textquotedblright, Università di Padova, via Trieste, 63, Padova, 35131, Italia \\
              Istituto Nazionale di Alta Matematica \textquotedblleft Francesco Severi\textquotedblright, Piazzale Aldo Moro, 5, Roma, 00185, Italia}
\email{ludovico.brunibruno@unipd.it}

\author{Federico Piazzon}
\address{Dipartimento di Matematica \textquotedblleft Tullio Levi-Civita\textquotedblright, Università di Padova, via Trieste, 63, Padova, 35131, Italia}
\email{fpiazzon@math.unipd.it}

\begin{document}
\maketitle

\begin{abstract}
    In this work we introduce the concept of admissible integral $k$-mesh for sampling differential forms with contiuous coefficients on a real body $E\subset \R^n$, and provide two techniques for the construction of admissible integral $k$-meshes on real bodies enjoying the Markov or the Bernstein inequality. Admissible integral $k$-meshes allow for the construction of robust approximation schemes, and are used to extract interpolation sets with high stability properties. To this end, the concepts of Fekete currents and Leja sequences of currents are formalized, and a numerical scheme for their approximation is proposed.
\end{abstract}

\section{Introduction}

Approximation of differential forms is an abstract framework that unifies many approximation procedures under the same perspective. Among them, we mention nodal interpolation \cite{Davis}, histopolation \cite{BruniErb23}, fitting \cite{Bjork} and approximation of vector fields \cite{ABRFace}; in fact, all the just mentioned procedures involve realizations of different orders of the differential forms. Obviously,  numerical methods for the resolution of partial differential equations based on such kind of approximations carry a formulation in the language of differential forms as well \cite{Hiemstra}. Further, this approach allows to gather many a priori different situations under the same perspective \cite{ArnoldFalkWintherActa}.

A formal framework for the approximation of differential forms has been proposed in \cite{BrPi25}, where the concept of Lebesgue constant defined for the interpolation of Whitney forms \cite{AlonsoRapettiLebesgue} was extended to generic differential forms and related with the common nodal concept. In \cite{BrPi25} the authors work in the space of differential forms with continuous coefficients over a real body $E$ (i.e., the closure of a bounded domain in $\R^n$) endowed by the $0$-norm $\|\cdot\|_0$, see \eqref{0normdef}. This choice is rather natural, as this norm both can be seen as a generalization of the uniform norm of functions, and it turns the space of differential $k$-forms with continuous coefficients into a Banach space, hereafter denoted by $\testforms(E)$. Perhaps even more importantly, the $0$-norm is defined by integration, the most natural operation on differential forms. This fact strongly justifies such a choice, as it makes the $0$-norm particularly tailored for working with the type of sampling operators considered both in \cite{BrPi25} and in the present work: normalized $k$-rectifiable currents with constant multiplicity and finite mass \cite{Krantz}. These functionals are, roughly speaking, integrals over oriented sets of constant measure theoretic dimension $k$, divided by the $k$-dimensional volume of the support of integration. We refer the reader to Section \ref{sec:2.2} below for the precise definitions of these mathematical objects.

The idea of the present work is to bridge this theoretical framework with a mathematical technology recently developed in multivariate polynomial approximation, namely \emph{polynomial admissible meshes}. 
Admissible meshes have been first introduced in \cite{CalviLevenberg08} to investigate the uniform convergence of discrete least squares approximation of continuous functions on a compact set by algebraic polynomials. Intuitively, the definition of admissible meshes is given by pairing a sampling inequality, somehow reminiscent of Marcinkiewicz–Zygmund inequality, with a control on the cardinality of the considered norming sets. Precisely, a sequence $\{\mathcal T^{(r)}\}_{r\in \N}$ of finite subsets of $E$ is termed admissible mesh for $E$ if $r\mapsto \Card \mathcal T^{(r)}$ has sub-exponential growth, and there exists $C\in ]0,+\infty[$ such that 
$$\max_{x\in E}|p(x)|\leq C\max_{y\in \mathcal T^{(r)}}|p(y)|,\;\;\forall p \in \mathscr P_r(E),$$
where $\mathscr P_r(E)$ denotes the space of algebraic polynomial of degree at most $r$ restricted to $E$. 

Admissible polynomial meshes are of leading interest in the construction of approximation strategies such as interpolation, discrete least squares, and approximate optimal designs.  Namely, they found significant applications %have been successfully used in a number of topics of approximation theory and related fields such as, e.g., 
in the construction of quasi-optimal interpolation arrays  \cite{BosDeMarchiSommarivaVianello10,BosCalviLevenbergSommarivaVianello11}, in the analysis of asymptotic behavior of multivariate orthogonal polynomials and pluripotential theory \cite{Piazzon19}, in the field of quadrature and optimal experimental designs \cite{PiazzonSommarivaVianello17,BosPiazzonVianello20}, and polynomial optimization \cite{PiazzonVianello19}. %Due to this variety of applications, 
For, the algorithmic construction of admissible polynomial meshes has been extensively studied, borrowing various techniques and tools from %different theories as 
classical polynomial inequalities, convex geometry, and pluripotential theory %, see e.g., 
\cite{Kroo11,Kroo19,Piazzon16,PiazzonVianello14}. The aim  of the present work is to extend this approach to the approximation of differential forms.

The first step in our construction is to introduce \emph{admissible integral $k$-meshes}. This is done in Section \ref{sec:4}, where, to keep the presentation fairly general, we start from a given increasing dense sequence $\{\spaceVr\}_{r\in \N}$ of subspaces of $\testforms(E)$, relevant examples of such subspaces  (see Subsection \ref{sect:polydf}) are polynomial differential forms, and trimmed polynomial differential forms. These spaces have been considered in finite element methods, see \cite{ArnoldFalkWintherActa} and references therein. In order to define admissible integral $k$-meshes, we first need to generalize the concept of mesh itself to the context of differential form and to adapt it to the use of $\|\cdot\|_0$. This leads to the introduction of \emph{integral $k$-mesh}, defined as a sequence $\{\mathcal T^{(r)}\}_{r\in \N}$ of finite collections $\mathcal T^{(r)}=\{T^{(s,r)}\}_{s=1}^{M(r)}$ of linear continuous functionals on $\testforms(E)$ that are representable by normalized integration along a compact piece of an affine variety of dimension $k$, see Definition \ref{def:integralkmesh}. Note that we can associate to any integral $k$-mesh a family of semi-norms $\|\omega\|_{\mathcal T^{(r)}}:=\max_{s=1,\dots,M(r)}|T^{s,r}(\omega)|$, so we can introduce \emph{admissible integral $k$-meshes} by analogy with admissible polynomial meshes. Precisely, we say that the integral $k$-mesh $\{\mathcal T^{(r)}\}_{r\in \N}$ is admissible for $\{\spaceVr\}_{r\in \N}$ if the sampling inequality 
\begin{equation}\label{eq:samplingproperty}
\|\omega\|_0\leq C\|\omega\|_{\mathcal T^{(r)}},\;\;\forall \omega\in \spaceVr
\end{equation}
holds for any $r\in \N$ and a constant $C$ not depending on $r$, and if the cardinality of $\mathcal T^{(r)}$ has sub-exponential growth, i.e., $\limsup_r[\Card \mathcal T^{(r)}]^{\frac 1 r}\leq 1.$

Section \ref{sec:construction} is entirely devoted to the construction of admissible integral $k$-meshes for the case of $\spaceVr=\forms{r}{k},$ the space of differential $k$-forms with polynomial coefficients of degree at most $r$. This task is accomplished by  extending the techniques used for the case $k=0$, i.e., polynomial admissible meshes. This generalization is non-trivial, as the geometric setting and the use of an integral norm and integral functionals cause a number of issues.

The constructions we present rely on classical multivariate polynomial inequalities of Markov (see, e.g., \cite{PaPl86}) and Bernstein \cite{Miro92} type. In Theorem \ref{FE-CM} we exploit Markov Inequality to provide a  sufficient condition for a finite collection of linear functionals $\mathcal T=\{T^1,\dots,T^M\}$ (of the same type as in the definition of integral $k$-meshes) to enjoy a sampling inequality as \eqref{eq:samplingproperty}. This result allows for an algorithmic construction of a $\forms{r}{k}$-admissible integral $k$-mesh for any convex body, see Proposition \ref{prop:convmesh}. It is worth noticing that Theorem \ref{FE-CM} may be generalized to the non-convex setting following the ideas of \cite{CalviLevenberg08}, provided $E$ still enjoys a Markov Inequality, this would extend the results of Proposition \ref{prop:convmesh} to a fairly general class of real bodies.  
In contrast, the use of Markov Inequality, as it does in the case $k=0$, would typically produce $\forms{r}{k}$-admissible integral $k$-meshes with very rapidly increasing cardinality as $r$ grows large. This makes the practical use of such admissible integral $k$-meshes feasible only for mild values of $r$.

To partially overcome such an issue, we also present a construction relying on the multivariate version of the classical Bernstein Inequality, namely the Baran Inequality \cite{Miro92,Miro94}. In Theorem \ref{th:FEB} we provide a sufficient condition for the sampling inequality  \eqref{eq:samplingproperty} to hold, expressed in terms of the so-called Baran distance. This distance is constructed in \cite{LenNormShayne04} via Carnot-Caratheodory construction starting from the Baran metric, which in general is just a Finsler metric. We point out that the Baran metric of a given generic convex body is not a priori known, and is not in general Riemannian. These facts significantly restrict the field of applicability of Theorem \ref{th:FEB}. However, when Theorem \ref{th:FEB} can be used instead of Theorem \ref{FE-CM} for constructing  $\forms{r}{k}$-admissible integral $k$-meshes, the resulting cardinality dramatically drops: intuitively speaking the cardinality of a $\forms{r}{k}$-admissible integral $k$-mesh constructed by Baran Inequality grows as the square root of the cardinality of the one constructed relying on the Markov Inequality. Remarkably, relevant sets as balls, cubes, and simplexes fall into the category for which Theorem \ref{th:FEB} is applicable, \cite{Pi19}.  In Subsection \ref{subsec:simplex} we present an algorithmic construction of an integral $k$-mesh for the $2$-simplex and the $3$-simplex (see Algorithm \ref{alg:meshthesimplex}). Exploiting Theorem \ref{th:FEB} and the explicit formula for the Baran distance in the simplex, we are able to prove that the integral $k$-mesh computed by Algorithm 1 is indeed $\forms{r}{k}$-admissible, see Proposition \ref{prop:provealg}.

The last two sections of the manuscript are concerned with applications of the theory developed above to the framework of approximation and interpolation. In Section \ref{sec:5} we consider the generalized weighted least squares approximation over admissible integral $k$-meshes. Namely, we study the sequence of projection operators $P^{(r)}:\testforms\rightarrow \spaceVr$ defined by 
$$P^{(r)}(\omega):=\argmin_{\theta\in \spaceVr}\sum_{s=1}^{M(r)}w^{(s,r)}|T^{(s,r)}(\omega-\theta)|^2,$$
where $w^{(s,r)}$ are positive weights, $\{\spaceV_r\}_{r\in\N}$ is a dense sequence of subspaces of $\testforms(E)$, and $\{\mathcal T^{(r)}\}_{r\in \N},$ $\mathcal T^{r}=\{T^{(s,r)}\}_{s=1}^{M(r)}$ is a $\{\spaceV_r\}_{r\in\N}$-admissible integral $k$-mesh. Let us recall that the Lebesgue constant $\L^{(r)}$ of this approximation scheme has been defined in \cite{BrPi25}, where the authors also show that $\L^{(r)}$ always dominates the operator norm of $P^{(r)}$, thus $\L^{(r)}$ measures both the stability and the effectiveness (via the Lebesgue Inequality) of the approximation procedure. We state in Proposition \ref{prop:errorestimatefittingmesh} an upper bound for  $\L^{(r)}$, and the classical error bound for $\|\omega-P^{(r)}\omega\|_0$ follows. Following \cite{CalviLevenberg08} and exploiting the sampling property of admissible integral $k$-meshes, it is possible to sharpen this result and obtain a revisited error estimate, see Proposition \ref{corrollaryfitting}.

In Section \ref{sec:interp} we consider interpolation of differential forms by differential forms with polynomial coefficients. We adapt the definitions of Fekete arrays and Leja sequences to the context of generalized interpolation by means of the operators $\Pi^{(r)}:\testforms(E) \rightarrow \forms{r}{k}$ implicitly defined by
$$T^{(s,r)}(\Pi^{(r)}\omega)=T^{(s,r)}(\omega),\;s=1,2,\dots, N(r):=\ddim\forms{r}{k},$$
where $\{T^{(s,r)}\}_{s=1}^{N(r)}$ is a \emph{unisolvent} set of currents. We point out that being unisolvent is by itself a remarkable property of a sets of currents, while Fekete arrays of currents not only are unisolvent, they also have a Lebesgue constant that grows at most as $\ddim \forms{r}{k}$. The construction of Fekete arrays (or Leja sequences) is a completely unfeasible task, being the complexity of the problem exponentially increasing as $r\to +\infty$. In the context of nodal interpolation of functions admissible polynomial meshes have been successfully used in order to discretize the continuous optimization problems that defines Fekete arrays and Leja sequences. We extend this approach to the framework of differential $k$-forms and admissible integral $k$-meshes: Theorem \ref{th:lebesgueforAFP} asserts that the Fekete arrays $\mathcal F^{(r)}=\{F^{1,r},\dots,F^{(N(r),r)}\}$ extracted from an admissible integral $k$-mesh still has a Lebesgue constant enjoying
$$\L^{(r)}\leq C N(r)\,,$$
where $C$ is the constant appearing in \eqref{eq:samplingproperty}. This is the first example in the literature of a generalized interpolation scheme for differential forms relying of integration functionals and having a Lebesgue constant with provable \emph{polynomial growth}. 

We stress that the algorithms AFP (Approximate Fekete Points) and DLS (Discrete Leje Sequences) \cite{BosDeMarchiSommarivaVianello10,BosCalviLevenbergSommarivaVianello11}, proposed in the nodal framework (i.e. $k=0$) for the extraction of Fekete or Leja points from an admissible mesh, do not need any modification to be applicable in the framewok of the present study. Indeed, starting from an integral admissible $k$-mesh, \emph{good} (i.e., with slowly increasing Lebesgue constant) unisolvent sets of currents can be extracted from it in polynomial time by directly applying the AFP algorithm.

 % Introduction
\section{Admissible integral $k$-meshes}\label{sec:4}
The aim of this present section is to introduce the main mathematical object the present paper is concerned on, that is admissible integral $k$-meshes. To this aim, we first need to briefly recall few definitions and results regarding exterior algebra, differential forms, and currents. This will also allow us to fix the notation we will use later on.  
\subsection{Background and notation}\label{sec:2.2}
We use the symbol $\Lambda_k$ for $k$-vectors, i.e.,  the $k$-th exterior power of $\R^n$. This set is the $\binom{n}{k}$-dimensional real vector space spanned by all the wedge products of the form $ \e_{\alpha}: = e_{\alpha(1)} \wedge \ldots \wedge e_{\alpha(k)} $ with $\alpha$ increasing multi-index. In $\Lambda_k$ we will often use the subclass of \emph{simple} $k$-vectors, i.e., $k$-vectors $\tau$ of the form $\tau=\tau_1\wedge\tau_2\wedge\dots\wedge\tau_k$ for a suitable choice of $\tau_1,\dots,\tau_k\in \R^n$.
We endow $\Lambda_k$ with the Euclidean scalar product $(\cdot,\cdot)_{\Lambda_k}$, which is defined by declaring $\e_\alpha$'s an orthonormal basis, see, e.g., \cite[p. 16]{Bhatia}, inducing the Euclidean norm $|\cdot|_{\Lambda_k}.$ We recall that, if $\tau=\tau_1\wedge\tau_2\wedge\dots\wedge\tau_k$ and $A=[\tau_1,\tau_2,\dots,\tau_k]\in M_{n,k}(\R)$, then, by Cauchy-Binet Formula, one has
$$|\tau|_{\Lambda_k}^2=\det(A^tA)={\sum_{|\alpha|=k}}'{\det}^2(A^\alpha),$$
where $(A^\alpha)_{i,j}=A_{\alpha(i),j}$, $i,j=1,\dots,k$, and, as it is customary, the symbol $\sum_{|\alpha|=k}'$ denotes summation along increasing multi-indices of length $k$.

Similarly, $ \Lambda^k $ denotes the $k$-th exterior power of the space of linear forms on $\R^n$, the Euclidean scalar product $(\cdot,\cdot)_{\Lambda^k}$ on $\Lambda^k$ can be defined by declaring $\e^\alpha:=e_{\alpha(1)}^*\wedge\dots\wedge e_{\alpha(k)}^*$'s to be orthonormal. Note that $\Lambda^k$ is indeed the dual space of $\Lambda_k$, being $\langle\omega_1\wedge\dots\wedge\omega_k; \, v_1\wedge\dots\wedge v_k\rangle=\det([\omega_i(v_j)]_{i,j=1,\dots,k})$
the dual pairing.

On $\Lambda^k$, apart from the Euclidean norm, we consider also the \emph{comass norm} $|\cdot|^*$, defined by
\begin{equation}\label{eq:comassnorm}
|\omega|^*:=\{\langle\omega;\tau\rangle:\,\tau\in \Lambda_k\text{ simple }, |\tau|_{\Lambda_k}\leq 1\}\,.
\end{equation}

A differential $ k $-form $ \omega $ on $ \R^n $ is a mapping that associates to each $ x \in \R^n $ a $k$-covector of $ \Lambda^k $. This mathematical object is inextricably intertwined with integration, as much in mathematical theories as in physical applications. In the present paper we will always work with a rather generalized concept of integration of differential forms, only requiring that the domain of integration is a $k$-rectifiable set $S$ (see, e.g., \cite[\S 5.4]{Krantz}) of locally finite Hausdorff measure endowed by an orientation $\tau$ \cite[Def. 7.5.1 (3)]{Krantz}, i.e., any unit simple $k$-vector field $\tau=\tau_1\wedge\dots\wedge \tau_k\in \Lambda_k$ such that, for almost every $x\in S$, $\{\tau_1(x),\dots,\tau_k(x)\}$ is a basis of the approximate tangent space of $S$ at $x$ \cite[Def. 5.4.4]{Krantz}. Namely, we refer to the quantity
\begin{equation*}\label{eq:defineintegration}
\int_S \langle\omega(x);\tau(x)\rangle \de \haus^k(x)\, .
\end{equation*}
as the \emph{integral} of $\omega$ over $S$ with respect to the orientation $\tau$. Here and throughout the paper we denote by $\haus^k$ the $k$-dimensional Hausdorff measure on $\R^n$.

Let $ E $ be a \emph{real body}, i.e., the closure of a bounded domain of $\R^n$. For any $k\in \{0,1,\dots,n\}$ we denote by $\testforms(E):=\left(\mathscr C^0(E,{\Lambda}^k),\|\cdot\|_0\right)$ the Banach space of \emph{test forms} of order $k$ with continuous coefficients over $E$ endowed by the norm
\begin{equation}\label{0normdef}
\|\omega\|_0:=\sup\frac 1{\haus^k(S)}\int_S \langle\omega(x);\tau(x)\rangle \de \haus^k(x)\,,
\end{equation}
where the supremum is taken among all oriented $k$-rectifiable sets $S\subset E$ of positive finite $\haus^k$ measure. As customary, we will often refer to $\|\cdot\|_0$ as the \emph{$0$-norm}.

It is worth stressing here that, when $E$ is a real body, the $0$-norm is the natural extension of the uniform norm of function to the context of differential forms. In fact, it is not difficult to prove that it agrees with the uniform norm of the point-wise comass norm of the differential form:
\begin{equation}
\|\omega\|_0=\max_{x\in E}|\omega(x)|^*,\;\forall \omega\in \testforms(E)\,.
\end{equation}
Note that  the comass norm $|\cdot|^*$ is somehow reminiscent of the role of integration appearing in \eqref{0normdef}. Indeed, on one hand  simple $k$-vectors (used as test vectors in the definition \eqref{eq:comassnorm} of $|\cdot|^*$) are the only elements of $\Lambda_k$ that can be thought as oriented basis of a $k$-dimensional linear subvariety of $\R^n$, hence as basis of the tangent space of a $k$-dimensional smooth manifold at a point. On the other hand, when integrating a differential form on an oriented manifold (or on a $k$-rectifiable oriented set), $\omega(x)$ will contribute to the integral only through the pairing with simple $k$-vectors.

The topological dual $\currents(E):=(\testforms(E))'$ of $\testforms(E)$ is the space of \emph{currents of order zero} and dimension $k$. This space is naturally endowed by the operator norm  
\begin{equation}\label{eq:mass}
\Mass(T):=\sup\{|T(\omega)|, \; \omega\in \mathscr D_0^k(E),\;\|\omega\|_0\leq 1\},
\end{equation}
which is referred to as the \emph{mass} of the current $T$. 
%
%If $T\in \currents(E)$, the support $\support T$ of the current $T$ is defined as the complement in $E$ of the set where $T$ vanishes identically, i.e.,
%$$\support T:=\left(\bigcup\{A:\,A\text{ is open, }T(\omega)=0\,\forall \omega\in\testforms(A)\}\right)^c.$$
 In what follows we will consider in $\currents(E)$ the subclass $\mathcal I^k(E)$ of \emph{currents of integration}, i.e., any current $[S,\tau]$ defined by integration over a $k$-rectifiable set $S$ of finite measure and orientation $\tau$, i.e.,
\begin{equation}\label{integralcurrent}
[S,\tau](\omega)=\int_S\langle\omega;\tau\rangle d\haus^k\,,
\end{equation} 
and the subclass $\averaging(E)$ of \emph{currents of integral averaging}
\begin{equation}
\averaging(E):=\left\{\frac{[S,\tau](\omega)}{\haus^k(S)},\;S\subset E\text{ is a }k\text{-rectifiable set of  with orientation } \tau,\;0<\haus^k(S)<+\infty \right\}
\end{equation}
Note that in particular one has $\Mass([S,\tau])=\haus^k(S)$, hence $\Mass(T)=1$ $\forall T\in \averaging(E)$ by construction. 
Let us remark that, unwinding the definitions of $\averaging(E)$ and $\|\cdot\|_0$, we get a characterization of $ \Vert \omega \Vert_0 $ in terms of averaging currents:
$$\|\omega\|_0=\sup_{T\in\averaging(E)}T(\omega),\quad \forall \omega\in \testforms(E).$$

Let $\varphi:\widehat E\rightarrow E$ is a $\mathscr C^1$ diffeomorphism. Then we can introduce the push-forward $\varphi_*\widehat T$ of a current $\widehat T\in \currents(\widehat E)$ by setting $\varphi_*\widehat T(\omega)=T(\varphi^*\omega)$, $\forall \omega\in \testforms(E)$, where $\varphi^*:\testforms(E)\rightarrow \testforms(\widehat E)$ is the pull-back of differential forms. It is worth noticing here that the class $\averaging(E)$ is not stable under push-forward operations \cite{BrPi25}, since the mass of a current may vary under such operation.  

\subsection{Introducing admissible integral $k$-meshes}\label{sec:4.1}
%It is rather clear that the definition of admissible meshes has been designed around the uniform norm and considering point-wise evaluation functionals as sampling operator, these being the key ingredients in classical polynomial interpolation of functions. Note that such a setting can be casted within the framework of the present paper by setting $k=0.$  

The classical theory of admissible meshes has been designed around the uniform norm and considering point-wise evaluation functionals as sampling operators. In the present subsection, drawing inspiration from the nodal case, we generalize the definition of admissible polynomial meshes to all cases $k\in\{1,\dots,n\}$. As one may expect, the most impactful changes concern the replacement of the uniform norm by $\|\cdot\|_0$, and the use of integral averages as sampling operators. 

Let us start generalizing the concept of mesh to the setting of differential $k$-forms: the idea is to consider sequences of finite collections of currents of integral averaging, each current being supported on a compact piece (lying in $E$) of an affine variety of dimension $k$. Formally, we introduce the following:
\begin{definition}[Integral $k$-mesh]\label{def:integralkmesh}
Let $E\subset \R^n$ be a compact set with non-empty interior and let $k\in \N$. Let, for any $r\in \N$ and any $s=1,2,\dots, M(r)<+\infty$, $x^{(s,r)}\in E$, $\Omega^{(s,r)}\subset\R^k$ a compact set with $0<\leb^k(\Omega^{(s,r)})$, and $A^{(s,r)}=[A_{:,1}^{(s,r)},\dots, A_{:,k}^{(s,r)}]\in M_{n,k}(\R)$, $\Rank(A^{(s,r)})=k$ such that
$$S^{(s,r)}:=\{x^{(s,r)}+A^{(s,r)}y,y\in \Omega^{(s,r)}\}\subset E.$$
Let, for any $r\in \N$,
$$\mathcal T^{(r)}:=\{T^{(1,r)},\dots,T^{(M(r),r)}\},\;\;\;$$
where, for $s=1,2,\dots, M(r)$,
$$T^{(s,r)}:= \frac{[S^{(s,r)},\sigma^{(s,r)}]}{\haus^k(S^{(s,r)})} \quad \text{and} \quad \sigma^{(s,r)}:=\frac{A_{:,1}^{(s,r)}\wedge\dots\wedge A_{:,k}^{(s,r)}}{|A_{:,1}^{(s,r)}\wedge\dots\wedge A_{:,k}^{(s,r)}|}\,.$$
Then the sequence $\{\mathcal{T}^{(r)}\}_{r\in \N}$ of finite subsets of $\averaging(E)$ is termed \emph{integral $k$-mesh}.
\end{definition}
By a slight abuse of notation and terminology, we can identify $\mathcal T^{(r)}$ with the set of oriented supports $\{(x^{1,r},A^{(1,r)},\Omega^{(1,r)}),\cdots,(x^{M(r),r},A^{(M(r),r)},\Omega^{(M(r),r)})\}$, and term the latter \emph{integral $k$-mesh} as well. 
\begin{remark}
We decided to include in the definition \label{def:integralkmesh} the requirement of $S^{(s,r)}$ being a piece of an affine variety. This is mainly motivated from the perspective of applications and simplicity of coding. In particular, in this setting one has 
\begin{align*}
T^{(s,r)}(\omega)=&  \frac 1{\int_{\Omega^{(s,r)}}|A_{:,1}^{(s,r)}\wedge\dots\wedge A_{:,k}^{(s,r)}|\de y}\int_{\Omega^{(s,r)}}\langle\omega(x^{(s,r)}+A^{(s,r)}y);\sigma^{(s,r)}\rangle|A_{:,1}^{(s,r)}\wedge\dots\wedge A_{:,k}^{(s,r)}| \de y\\
=&  \frac 1{\leb^k(\Omega^{(s,r)})}\int_{\Omega^{(s,r)}}\langle\omega(x^{(s,r)}+A^{(s,r)}y);\sigma^{(s,r)}\rangle \de y\,.
\end{align*}
\end{remark}

The sequence $\{\mathcal{T}^{(r)}\}_{r\in \N}$ naturally induces a sequence of seminorms on $\mathscr D^k_0(E)$ defined by
$$\|\omega\|_{\mathcal{T}^{(r)}} :=\max_{s\in\{1,\dots, M(r)\}}|T^{(s,r)}(\omega)|\;,\;\;\forall \omega\in \mathscr D_0^k(E)\,,$$
which is used to extend the definition of admissible polynomial meshes to differential forms.
\begin{definition}[(Weakly-)\mesh]\label{def:admissiblemeshes}
Let $\{\mathscr{V}_{(r)}\}_{r\in \N}$ be an increasing sequence of linear subspaces of $\testforms(E).$ Let $\{\mathcal{T}^{(r)}\}_{r\in \N}$ be an integral $k$-mesh for the compact set $E\subset\R^n$. Then $\{\mathcal{T}^{(r)}\}_{r\in \N}$ is termed a $\{\mathscr{V}^{(r)}\}_{r\in \N}$-\emph{\mesh} if, in the above notation, 
\begin{align}
&\limsup_{r\to +\infty}\left[\Card(\mathcal{T}^{(r)})\right]^{1/r}\leq1 %\label{subexponentialgrowthHk} 
\notag\\
&C:=\sup_{r\in \N}\;\sup\left\{\frac{\|\omega\|_0}{\|\omega\|_{\mathcal{T}^{(r)}}}, \ \omega\in \mathscr{V}^{(r)}\setminus\{0\}\right\}<+\infty\,.\label{constantofmeshHk}
\end{align}
If the property \eqref{constantofmeshHk} is replaced by the weaker assumption
\begin{equation*}
\limsup_rC_r^{1/r}:=\limsup_r\left( \sup\left\{\frac{\|\omega\|_0}{\|\omega\|_{\mathcal{T}^{(r)}}}, \ \omega\in \mathscr{V}^{(r)}\setminus\{0\}\right\} \right)^{1/r}\leq 1\,,
\end{equation*}
then $\{\mathcal{T}^{(r)}\}_{r\in \N}$ is termed a $\{\mathscr{V}^{(r)}\}_{r\in \N}$-\emph{weakly admissible integral $k$-mesh}.
\end{definition}

In the following Lemma \ref{lem:basicproperties} we collect some basic properties of integral $k$-meshes, which we invite the reader to compare with \cite[\S 4]{BlChLe92}. We point out that some of the features of weakly admissible ($0$-)meshes do not have natural counterparts in the context of integral $k$-meshes with $k>0$.

\begin{lemma}\label{lem:basicproperties} The following properties of admissible integral $k$-meshes hold true.\\
\underline{Subspace}: if $\{\mathcal T^{(r)}\}$ is a $\{\mathscr V^r\}$-\mesh and $\mathscr W^r\subset \mathscr V^r$ for any $r$, then  $\{\mathcal T^{(r)}\}$ is a $\{\mathscr W^r\}$-\mesh as well. \\
\underline{Finite unions}: let $E=\cup_{i=1}^m E_i$ and denote by $\mathscr{V}^{(r)}_i$ the space of the restrictions of the elements of $\mathscr{V}^{(r)}$ to $E_i$. If, for any $i=\{1,\dots.m\}$, $\{\mathcal{T}_i^{(r)}\}_{r\in \N}$ is a $\{\mathscr{V}^{(r)}_i\}_{r\in \N}$-\mesh for $E_i$, then $\{\cup_{i=1}^m \mathcal T^{(r)}_i\}_{r\in \N}$ is a $\{\mathscr{V}^{(r)}\}_{r\in \N}$-admissible integral $k$-mesh for $E$.\\
\underline{Good unisolvent triangular arrays in $\averaging(E)$}: assume that, for any $r\in \N$, $\mathcal T^{(r)}$ is an interpolation set (i. e. $\Card \mathcal T^{(r)}=N(r):=\ddim \mathscr{V}^{(r)}$), and that
$$\limsup_r[\L(\mathcal T^{(r)}, \mathscr V^r)]^{1/r}\leq 1\,$$
where $\L(\mathcal T^{(r)}, \mathscr V^r)$ denotes the Lebesgue constant of $\mathcal T^{(r)}$ as defined in \cite{BrPi25} (see also \eqref{} in Section \ref{} below).Then $\{\mathcal{T}^{(r)}\}_{r\in \N}$ is a $\{\mathscr{V}^{(r)}\}_{r\in \N}$-weakly admissible integral $k$-mesh.\\
\underline{Push-forward by $\mathscr C^1$-diffeomorphism}, let $\varphi:\widehat E\rightarrow E$ a $\mathscr C^1$ diffeomorphism. Let $\{\mathscr V^r(\widehat E)\}_{r\in\N}$ an increasing sequence in $\testforms(\widehat E)$ and define, for any $r\in \N$, $\mathscr V^r(E):=\{\omega\in \testforms(E): \varphi^*\omega\in \mathscr V^r(\widehat E)\}$. If $\{\widehat{\mathcal T}^{(r)}\}_{r\in \N}$, $\widehat{\mathcal T}^{(r)}=\{ \widehat T^{(1,r)}, \dots,\widehat T^{(M(r),r)}\}$, is a weakly admissible integral mesh for $\{\mathscr V^r(\widehat E)\}_{r\in \N}$ with constant $\widehat C_r$, then, the sequence $\{\mathcal T^{(r)}\}_{r\in \N}$ defined by 
$$\mathcal T^{(r)}=\left\{ \frac{\varphi_* \widehat T^{(1,r)}}{\Mass(\varphi_* \widehat T^{(1,r)})}, \dots,\frac{\varphi_* \widehat T^{(M(r),r)}}{\Mass(\varphi_* \widehat T^{(M(r),r)})}\right\}$$ 
is a $\{\mathscr{V}^{(r)}(E)\}_{r\in \N}$-weakly admissible integral $k$-mesh of constant 
$$C_r\leq \left\|\prod_{j=1}^k \sigma_j^\varphi  \right\|_{\mathscr C^0(\widehat E)}\left\|\frac{1}{\prod_{j=n-k+1}^n \sigma_j^\varphi}  \right\|_{\mathscr C^0(\widehat E)} \widehat C_r\,,$$
where $\{\sigma_1^\varphi,\dots,\sigma_n^\varphi\}$ are the singular value functions of the differential of $\varphi$ arranged in a non-increasing order, and we denoted by $\|\cdot\|_{\mathscr C^0(\widehat E)}$ the uniform norm on continuous functions on $\widehat E$.
%$\widehat{\mathcal T}^{(r)}=\{T_{\widehat S_1},\dots,T_{\widehat S_{M(r)}}\}$ be an \mesh of constant $\widehat C$ for the compact set $\widehat E$. Let $\varphi: \widehat E\rightarrow E:=\varphi(\widehat E)$ be a non degenerate affine mapping with linear part having singular values $\sigma_1\geq \ldots \geq \sigma_n.$ %Then $\mathcal T^{(r)}:=\{T_{S_1},\dots,T_{S_{M(r)}}\}$, where $S_i:=\varphi(\widehat S_i)$, is an \mesh of constant
%Then $\mathcal T^{(r)}:=\{T_{\varphi(\widehat S_1)},\dots,T_{\varphi( \widehat {S}_{M(r)})}\}$ is an \mesh whose constant $ C $ satisfies
%$$C\leq\frac{\prod_{j=1}^k\sigma_j}{\prod_{j=n-k+1}^n\sigma_j}\widehat C\,.$$
\end{lemma}
\begin{proof}
The first three properties follows immediately from Definition \ref{def:admissiblemeshes}. The fourth follows as a corollary from \cite[Lemma 7]{BrPi25}.
\end{proof}
%\begin{align*}
%\|\omega\|_0= &\sup_{S\in \mathcal S^k(E)}\frac 1{\haus^k(S)}\left |\int_S\omega\right|=\sup_{\varphi (\widehat S) \in \mathcal S^k(E)}\frac 1{\haus^k(\varphi (\widehat S))}\left |\int_{\varphi (\widehat S)}\omega\right| =\sup_{\widehat S\in \mathcal S^k(\widehat E)}\frac 1{\haus^k(\varphi( \widehat S))}\left |\int_{\widehat S}\varphi^*\omega\right|\\
%=&\sup_{\widehat S\in \mathcal S^k(\widehat E)}\frac{\haus^k(\widehat S)}{\haus^k(\varphi(\widehat S))}|T_{\widehat S}(\varphi^*\omega)|\leq \sup_{\widehat S\in \mathcal S^k(\widehat E)}\frac{1}{\prod_{j=n-k+1}^n\sigma_j}|T_{\widehat S}(\varphi^*\omega)|=\frac{1}{\prod_{j=n-k+1}^n\sigma_j}\|\varphi^*\omega\|_0\\
%\leq & \frac{\widehat C}{\prod_{j=n-k+1}^n\sigma_j}\|\varphi^*\omega\|_{\widehat{\mathcal T}^{(r)}}=\frac{\widehat C}{\prod_{j=n-k+1}^n\sigma_j}\max_i |T_{\widehat S_i}(\varphi^*\omega)|\\
%=&\frac{\widehat C}{\prod_{j=n-k+1}^n\sigma_j}\max_i \frac{\haus^k(\varphi(\widehat S_i))}{\haus^k(\widehat S_i)}|T_{\varphi (\widehat S_i)}(\omega)|\leq \frac{\prod_{j=1}^k\sigma_j}{\prod_{j=n-k+1}^n\sigma_j}\widehat C\|\omega\|_{\mathcal T^{(r)}}\,.
%\end{align*}
%The claim is proved. \myqed
%\end{proof}

\subsection{Spaces of polynomial differential forms} \label{sect:polydf}
The concept of admissibility for integral $k$-meshes given in Definition \ref{def:admissiblemeshes} applies to any increasing sequence $\{\mathscr{V}^{(r)}\}_{r\in \N}$ of subspaces of $\testforms(E)$, but clearly the construction of specific admissible meshes depends on the considered sequence. In the present subsection we introduce two important instances of such sequences: %\fede{Inserire qui frasetta sulla rilevanza di questi spazi con citazioni di Arnold o chi per lui}: 
the space of complete polynomial forms (denoted by $\{\P_r \Lambda^k\}_{r\in \N}$)  and that of trimmed polynomial differential forms (denoted by $\{\P_r^- \Lambda^k\}_{r\in \N}$). The results of Sections \ref{sec:4} and \ref{sec:5} will be specialized to these settings. Such spaces are not only the natural counterpart of polynomials in the framework of differential forms, but are also celebrated finite element exterior calculus \cite{ArnoldFalkWintherActa}, and paved the way for a theory of interpolation of differential forms \cite{AlonsoRapettiLebesgue,ABRCalcolo}.

Let us denote by $ \P_r $ and $ \H_r $ the space of $n$-variate polynomials of total degree $ r $ and the space of homogeneous polynomial of degree $ r $, respectively. Let us also define the ring of polynomials as $ \P := \oplus_{r=0}^\infty \H_r $. On each $ \H_r $, we may consider the lexicographically ordered monomial basis $ \homonomials{r} $. This choice induces on the monomials $ x^\beta := \prod_{i=1}^n x_i^{\beta_i} $ the graded lexicographical order. We denote by $ \monomials{} $ such a basis for $ \P $, and by $ \monomials{r} $ the basis for $ \P_r $ consisting of the first $ \binom{n+r}{r}$ elements of $ \monomials{} $. %Further, this choice induces an obvious ordering on bases for $ \mathscr H_r = \mathscr P_r \setminus \mathscr{P}_{r-1} $.

The space of \emph{complete} polynomial differential forms is
\begin{equation} \label{eq:defPLk}
\P \Lambda^k := \bigoplus_{j=0}^\infty \H_j \otimes \Lambda^k = \P \otimes \Lambda^k ,
\end{equation}
that is, polynomial sections of the $k$-th exterior power of the cotangent bundle of $ \R^n $. Truncating the above direct sum to $ j = r $, we obtain the spaces $ \P_r \Lambda^k $, whose dimension is readily computed as 
\begin{equation*}
    N(r) := \dim \P_r \Lambda^k = \binom{r+n}{r} \binom{n}{k} .
\end{equation*}
Notice that, under the assumption we made on $E$, the space $ \P_r \Lambda^k (E) $ is an $ N(r) $-dimensional subspace of $\testforms{(E)}$. Since $ \Lambda^k = \Span \{ \de x_\alpha , \, | \alpha | = k , \,  \alpha \text{ increasing} \} $, for each $ j $ we may define a basis $ \khomonomials{j}{k} := \{ h \otimes \de x_\alpha, \, h \in \mathcal{B}^{\mathrm{hom}}_j, \, |\alpha|=k, \alpha \text{ increasing} \}$ for $ \H_j \otimes \Lambda^k $, where we again consider the lexicographical order on the pair $(h,\de x_\alpha)$. Note also that $\{\khomonomials{j}{k}\}_{j\in \N}$ induces a graded basis $ \kmonomials{}{k} $ for $ \P \Lambda^k $ via \eqref{eq:defPLk}. Finally, truncating this basis at $ N(r) $ we obtain the basis $ \kmonomials{r}{k} $ for $ \P_r \Lambda^k $.

\begin{example} %Let us give, ease of the reader, a minimal example of the above definitions in the case $n=2$, $k=1$, and $r=1$:
In the case $n=2$, $k=1$, and $r=1$, the above sets are:
\begin{eqnarray*}
&\homonomials{0}=\{1\},\; & \khomonomials{0}{1}=\{\de x_1,\de x_2\},\\
&\homonomials{1}=\{x_1,x_2\},\; &\khomonomials{1}{1}=\{x_1\de x_1,x_1\de x_2,x_2\de x_1,x_2\de x_2\},\\
&\homonomials{2}=\{x_1^2,x_1x_2,x_2^2\},\; &\khomonomials{2}{1}=\{x_1^2\de x_1,x_1^2\de x_2,x_1x_2\de x_1,x_1x_2\de x_2,x_2^2\de x_1, x_2^2\de x_2\}\,,
\end{eqnarray*} 
so that
$$\kmonomials{2}{1}=\{\de x_1,\de x_2, x_1\de x_1,x_1\de x_2,x_2\de x_1,x_2\de x_2, x_1^2\de x_1,x_1^2\de x_2,x_1x_2\de x_1,x_1x_2\de x_2,x_2^2\de x_1, x_2^2\de x_2\} . $$
\end{example}

As a second relevant space that fits our framework, we consider the space of \emph{trimmed} polynomial differential forms. This space is defined as 
$$ \P_r^- \Lambda^k := \P_{r-1} \Lambda^k \oplus \kappa \left( \H_{r-1} \Lambda^{k+1} \right) ,$$
where $ \kappa: \testforms \to \mathscr{D}_0^{k-1} $ is the \emph{Koszul differential} \cite[p. 852]{Lang}, i.e., the contraction with the identity vector field. With respect to Cartesian coordinates, if $ \omega := p(x_1, \ldots, x_n) \de x_{\sigma(1)} \wedge \ldots \wedge \de x_{\sigma(k)} $, then
$$ \kappa \omega = \sum_{i=1}^k (-1)^i p(x_1, \ldots, x_n) x_{\sigma(i)} \de x_{\sigma(1)} \wedge \ldots \wedge {\de x}_{\sigma(i-1)} \wedge  {\de x}_{\sigma(i+1)} \wedge \ldots \wedge \de x_{\sigma(k)} .$$
In the lowest order, i.e. for $ r = 1 $, $ \P_1^- \Lambda^k (E)$ is the space of \emph{Whitney forms} \cite[p. 139]{Whitney}. For $ r > 1$, $ \P_r^- \Lambda^k $ is an intermediate space between $ \P_{r-1} \Lambda^k $ and $ \P_r \Lambda^k $. %The space $ \P_r^- \Lambda^k (E) $ is naturally included in $ \P_r \Lambda^k (E) $ for each $ r \in \N $. 
As a consequence, most of our results cast for the complete space directly apply to the trimmed space as well.

  % definition of meshes
\section{Construction of $\forms{r}{k}$-admissible integral $k$-meshes}\label{sec:construction}

\subsection{Constructing $\forms{r}{k}$-admissible integral $k$-meshes by Markov Inequality} \label{sect:markov}
The Markov Inequality, here recalled in the forthcoming Eq. \eqref{markovinequality}, is a classical tool in approximation theory. Because of its ductility, it has been investigated for long time and now presents broad scope connections with several branches of mathematics; for an account, we address the reader to \cite{PaPl86,BaGo71}. 

Let us recall that a polynomial determining compact set $E\subset\R^n$ admits a Markov inequality if there exist $\markov<+\infty$ (the \emph{Markov constant}) and $\beta<+\infty$ (the \emph{Markov exponent}) such that, for any polynomial $p$ of degree at most $r$, one has
\begin{equation}\label{markovinequality}
\max_{x\in E}|\nabla p(x)|\leq \markov r^\beta\max_{x\in E}|p(x)|\,.
\end{equation}
As proved in  \cite{Wilhelmsen74}, any fat convex body (i.e. $E=\overline{\interior E}$) admits a Markov inequality with exponent $\beta=2$ and constant $ \markov$ equal to the reciprocal of the minimum distance of two supporting hyperplanes.
 
\begin{theorem}[Fundamental estimate on convex bodies]\label{FE-CM}
Let $E \subset \R^n$ be a convex body with Markov constant $ \markov$. Let $\mathcal T=\{T^{(s)}\}_{s=1,\dots,M}\subset \averaging(E)$ be such that, for any $s=1,2,\dots, M$, there exist $x^{(s)}\in E$, $A^{(s)}\in M_{n,k}(\R^n)$ with $\Rank(A^{(s)})=k$, a compact set $\Omega^{(s)}\subset\R^k$ of positive finite Lebesgue measure, for which
$$T^{(s)}(\omega)= \frac 1{\haus^k(S^{(s)})}\int_{S^{(s)}}\langle\omega;\sigma^{(s)}\rangle \de \haus^k\;,\;\;\forall \omega\in \testforms(E)\,,$$
where $S^{(s)}$ is the $k$-rectifiable set $\{x^{(s)}+A^{(s)}y,y\in \Omega^{(s)}\}\subset E$ endowed by the orientation
$$\sigma^{(s)}:=\frac{A_{:,1}^{(s)}\wedge\dots\wedge A_{:,k}^{(s)}}{|A_{:,1}^{(s)}\wedge\dots\wedge A_{:,k}^{(s,r)}|_{\Lambda_k}}\,.$$
Let $r\in \N$, and let us assume that there exists $c_1<1$ such that, for any $x\in E$, we can find $s_1,s_2,\dots,s_m\in \{1,\dots,M\}$ such that
\begin{align}\label{FA1}
&\sup_{\tau\in \Lambda_k}\Bigg\{\min_{a\in \R^m}\Big\{|a|_1: \sum_{j=1}^m a_j \sigma^{(s_j)}=\tau\Big\}, |\tau|=1, \tau\text{ simple}\Bigg\}=:c_2<+\infty\\
&\max_{j=1,\dots,m}\max_{z\in S^{s_j}}|x-z|\leq c_1\left(c_2 \markov r^2\sqrt{\binom{n}{k}}\right)^{-1}\,.\label{FA2}
\end{align}
Then, for any $\omega\in \mathscr P_r\Lambda^k$ we have
\begin{equation}\label{samplinginequality}
\|\omega\|_0\leq \frac{c_2}{1-c_1} \|\omega\|_{\mathcal{T}}.
\end{equation}
\end{theorem}
\begin{proof}
Let $\omega\in \forms{r}{k}$ be any polynomial form. We prove that \eqref{samplinginequality} holds for $\omega$ under the assumption that there exist a point $\bar x\in E$ and a simple $k$-vector $\bar\tau\in \Lambda_k$ such that 
\begin{equation}\label{achievemax}
\max_{x\in E}|\omega(x)|^*=\langle\omega(\bar x);\bar \tau\rangle,
\end{equation}
the general case will readily follow by approximation. 

Let $m\in \N$, $s_1,s_2,\dots,s_m\in\{1,\dots,M\}$ and $c_1<1$ be such that
\begin{equation}\label{specifyass2}
\max_{j=1,\dots,m}\max_{z\in S^{s_j}}|\bar x-z|\leq c_1\left(c_2\markov r^2\sqrt{\binom{n}{k}}\right)^{-1},
\end{equation}
and let  $\bar a\in \R^m$ be any vector realizing
\begin{equation}\label{specifyass1}
\min_{a\in \R^m}\Big\{|a|_1: \sum_{j=1}^m a_j \sigma^{(s_j)}=\bar \tau\Big\}\,.
\end{equation}
%Now, starting from \eqref{achievemax} and using \eqref{specifyass1} we can write
Plugging \eqref{specifyass1} into \eqref{achievemax}, we compute
\begin{align}
&\max_{x\in E}|\omega(x)|^*=\langle\omega(\bar x);\bar \tau\rangle=\sum_{j=1}^m\bar a_j\langle \omega(\bar x);\sigma^{(s_j)}\rangle=\sum_{j=1}^m\bar a_j\frac{1}{\haus^k(S^{s_j})}\int_{S^{s_j}}\langle \omega(\bar x);\sigma^{(s_j)}\rangle \de\haus^k(x)\notag\\
=&\sum_{j=1}^m\bar a_j\frac{1}{\haus^k (S^{s_j})}\int_{S^{s_j}}\langle \omega(x);\sigma^{(s_j)}\rangle \de\haus^k(x)+\sum_{j=1}^m\bar a_j\frac{1}{\haus^k (S^{s_j})}\int_{S^{s_j}}\langle \omega(\bar x)-\omega(x);\sigma^{(s_j)}\rangle \de\haus^k(x)\notag\\
\leq& |\bar a|_1\|\omega\|_{\mathcal{T}}+|\bar a|_1\max_{j=1,\dots,m} \frac{1}{\haus^k (S^{s_j})}\int_{S^{s_j}}|\langle \omega(\bar x)-\omega(x);\sigma^{(s_j)}\rangle| \de\haus^k(x)\notag\\
\leq& c_2\left(\|\omega\|_{\mathcal{T}}+\max_{j=1,\dots,m} \frac{1}{\haus^k (S^{s_j})}\int_{S^{s_j}}|\omega(\bar x)-\omega(x)|_{\Lambda^k} \de\haus^k(x)\right)\notag\\
=& c_2\left(\|\omega\|_{\mathcal{T}}+\max_{j=1,\dots,m} \frac{1}{\leb^k(\Omega^{s_j})}\int_{\Omega^{s_j}}|\omega(\bar x)-\omega(x^{s_j}+A^{s_j}y)|_{\Lambda^k} \de y\right)\notag\\
=&c_2\left(\|\omega\|_{\mathcal{T}}+\max_{j=1,\dots,m} \frac{1}{\leb^k(\Omega^{s_j})}\int_{\Omega^{s_j}}\left({\sum_{\alpha}}'(\nabla \omega_\alpha(\xi_{\alpha,j,y});\bar x-x^{s_j}-A^{s_j}y)^2\right)^{1/2}\de y\right)\,,\label{estimate1}
\end{align}
%where $\xi_{\alpha,j,y}$ is a point in the segment connecting $\bar x\in E$ to $x^{s_j}+A^{s_j}y\in E$. Notice that $\xi_{\alpha,j,y} \in E$ since $ E $ is convex.
where the point $\xi_{\alpha,j,y}$ lies in the segment $[\bar x, x^{s_j}+A^{s_j}y]$. Notice that $\xi_{\alpha,j,y} \in E$ since $ E $ is convex.

Using the Cauchy-Schwartz Inequality, the bound of Eq. \eqref{specifyass2}, and the Markov Inequality \eqref{markovinequality}, we obtain that, for each $ y\in \Omega^{s_j}$ and $j\in\{1,\dots,m\}$, the following inequality holds:
\begin{align}
&\left({\sum_{\alpha}}'(\nabla \omega_\alpha(\xi_{\alpha,j,y}),\bar x-x^{s_j}-A^{s_j}y)^2\right)^{1/2}\leq \left({\sum_{\alpha}}'|\nabla \omega_\alpha(\xi_{\alpha,j,y})|^2\right)^{1/2}|\bar x-x^{s_j}-A^{s_j}y|\notag\\
\leq& \left({\sum_{\alpha}}'|\nabla \omega_\alpha(\xi_{\alpha,j,y})|^2\right)^{1/2} \max_{j=1,\dots,m}\max_{z\in S^{s_j}}|\bar x-z|\leq \left({\sum_{\alpha}}'|\nabla \omega_\alpha(\xi_{\alpha,j,y})|^2\right)^{1/2} c_1\left(c_2\markov r^2\sqrt{\binom{n}{k}}\right)^{-1}\notag\\
\leq& \left({\sum_{\alpha}}' \markov^2 r^4 \max_{x\in E}|\omega_\alpha(x)|^2 \right)^{1/2} c_1\left(c_2\markov r^2\sqrt{\binom{n}{k}}\right)^{-1} = \left({\sum_{\alpha}}' \max_{x\in E}|\omega_\alpha(x)|^2 \right)^{1/2} c_1\left(c_2\sqrt{\binom{n}{k}}\right)^{-1}\notag\\
\leq& \sqrt{\binom{n}{k}}\max_{\alpha} \max_{x\in E}|\omega_\alpha(x)|c_1\left(c_2\sqrt{\binom{n}{k}}\right)^{-1}=\frac{c_1}{c_2}\max_{\alpha} \max_{x\in E}|\omega_\alpha(x)|\leq \frac{c_1}{c_2}\max_{x\in E}|\omega(x)|^*\,.\label{estimate2}
\end{align}
Finally, using Eq. \eqref{estimate2} to estimate the last term of Eq. \eqref{estimate1}, we obtain
$$\max_{x\in E}|\omega(x)|^*\leq c_2\left(\|\omega\|_{\mathcal{T}}+\frac{c_1}{c_2}\max_{x\in E}|\omega(x)|^*\right) .$$
Since we are assuming $c_1<1$, Eq. \eqref{samplinginequality} follows. \myqed
\end{proof}

\begin{remark}
The convexity assumption of Theorem \ref{FE-CM} strongly simplifies computations, but is not strictly necessary. The extension to more general compact sets satisfying a Markov inequality can be obtained by a suitable modification of the argument in the proof, following the lines of \cite[Thm. 5]{CalviLevenberg08}.
\end{remark}

\subsubsection{An explicit example of integral $k$-mesh by Markov Inequality}
We show how Theorem \ref{FE-CM} can be used to construct an integral $k$-mesh on a real fat convex body $E$. %Indeed it is sufficient to adapt the construction of \cite[Th.5]{CalviLevenberg08} to our context. 

Let $0<c_1<1$ and $r\in \N$. Consider a bounding box $Q:=[a_1,b_1]\times [a_2,b_2]\times \dots \times[a_n,b_n]\supset E$ and a tessellation $\{Q_i\}_{i=1,\dots,\tilde N}$ of $Q$ made of coordinate $n$-dimensional cubes with side lengths $ d $ bounded from above by 
\begin{equation}\label{mydiameter}
d:=\frac{c_1w(E)}{4\sqrt{n} \binom{n}{k}r^2},\;w(E):=\min_{v\in \R^n}\min\{\ell: |(x;v)|\leq \ell\; \forall x\in E\}.
\end{equation}
For any $i=1,\dots,m$ such that $K_i:=Q_i\cap\interior E\neq \emptyset$, pick $\tilde x_i\in K_i$ and define
$$\{x^{(s,r)}\}_{s=1}^{M(r)}:=(\, \overbrace{\tilde x_1,\ldots,\tilde x_1}^{\binom{n}{k}\text{ times}}, \overbrace{\tilde x_2,\ldots,\tilde x_2}^{\binom{n}{k}\text{ times}},\ldots, \overbrace{\tilde x_m,\ldots,\tilde x_m}^{\binom{n}{k}\text{ times}}\,)\,.$$
Pick 
$$0<\epsilon^{(s,r)}\leq\min\left(\dist(x^{(s,r)},\partial E),\dist(x^{(s,r)},\partial Q_{i(s)})  \right),$$
where $i(s)$ is the unique index such that $x^{(s,r)}\in Q_{i(s)}$, and set $\Omega^{(s,r)}\equiv\epsilon^{(s,r)}\Sigma_k,\,\forall s=1,\dots,M(r).$
Given the enumeration $\alpha_1,\alpha_2,\dots,\alpha_{\binom{n}{k}}$ of the set of multi-indices $\alpha$ of length $k$ we set
\begin{equation*}
\label{tuttelemappelineari}
\{A^{1,r},A^{2,r},\dots,A^{\binom{n}{k},r},A^{\binom{n}{k}+1,r},\dots,A^{M(r),r}\}:=\{P^{\alpha_1},P^{\alpha_2},\dots,P^{\alpha_{\binom{n}{k}}},P^{\alpha_1},\dots,P^{\alpha_{\binom{n}{k}}}\},
\end{equation*}
where $P^\alpha$ is the $n$ by $k$ matrix whose columns are the vectors $e_{\alpha(1)},\dots,e_{\alpha(k)}$ of the canonical basis of $\R^n$. Finally consider the collection 
\begin{equation*}\label{myconvmesh}
\mathcal T^{(r)}:=\{(x^{(s,r)},\Omega^{(s,r)},A^{(s,r)}),s=1,\dots,M(r)\}.
\end{equation*} 
An example of this construction is reported in Figure \ref{fig:cassinimesh}. Note in particular that, since any $x\in E$ belongs to the closure of some $K_i$, we can find $\binom{n}{k}$ elements $(x^{(s_j,r)},\Omega^{(s_j,r)},A^{(s_j,r)})$ of $\mathcal T^{(r)}$ such that 
\begin{enumerate}[a)]
\item $x^{(s_j,r)}\equiv \tilde x_i$ for $j=1,2,\dots,\binom{n}{k}$;
\item The set of $k$-vectors $\tau_j=\frac{A^{(s_j,r)}_{:,1}\wedge A^{(s_j,r)}_{:,2}\dots\wedge A^{(s_j,r)}_{:,k}}{|A^{(s_j,r)}_{:,1}\wedge A^{(s_j,r)}_{:,2}\dots\wedge A^{(s_j,r)}_{:,k}|}$ is an orthonormal basis of $\Lambda_k$;
\item Setting $F^{(s_j)}:=\{x^{(s_j,r)}+A^{(s_j,r)}y, y\in \Omega^{(s_j,r)} \}$, we have $ \max_{j=1,\dots,\binom{n}{k}}\max_{z\in \mathcal F^{(s_j)}}|x-z|\leq \sqrt{n} d.$
\end{enumerate}

Also, note that if we consider the canonical basis $\{e_{\alpha(1)}\wedge\dots\wedge e_{\alpha(k)}:\,|\alpha|=k\}$ of $\Lambda_k(\R^n)$, then any unit simple $k$-vector $\tau$ can be written in the form $\tau=\sum'_{|\alpha|=k}a_\alpha e_{\alpha(1)}\wedge\dots\wedge e_{\alpha(k)}$ with
$$|a|_1\leq \sqrt{\binom{n}{k}}.$$ Thus, due to b) above, we can assume \eqref{FA1} with $c_2=\sqrt{\binom{n}{k}}.$

On the other hand, we recall that the Markov constant of $E$ is bounded by $4/w(E)$ due to the classical result in \cite{Wilhelmsen74}. Thus, combining \eqref{mydiameter} with property c) above, we have
$$ \max_{j=1,\dots,\binom{n}{k}}\max_{z\in \mathcal F^{(s_j)}}|x-z|\leq\sqrt{n}d=\frac{c_1w(E)}{4 \binom{n}{k}r^2}= \frac{c_1}{\markov r^2\binom{n}{k}},$$
i.e., \eqref{FA2} holds as well. 

Finally, we can give an asymptotic upper bound to $\Card \mathcal T^{(r)}$ noticing that the number of $k$-faces of the tessellation that lie $E$ is smaller that the number of $k$-faces of the tessellation. The latter is at most $\binom{n}{k}$ times the number of cubes, thus 
$$ \Card \mathcal T^{(r)}\leq \binom{n}{k}\left(\frac{\mathrm{diam}(E)}{d}\right)^n=\binom{n}{k}^{n+1}\left(\frac{\mathrm{diam}(E)}{w(E)} \frac{4\sqrt n}{c_1} \right)^n r^{2n}=\mathcal O(r^{2n})\,.$$

Since the hypotheses of Theorem \ref{FE-CM} are fulfilled, we can claim the following result.
\begin{proposition}[Admissible integral $k$-mesh on fat convex body]\label{prop:convmesh}
The sequence $\mathcal{T}=\{\mathcal{T}^{(r)}\}_{r\in \N}$ defined in \eqref{myconvmesh} is an \mesh for $E$ with constant $C:=\frac{\sqrt{\binom{n}{k}}}{1-c_1},$ i.e.
$$\|\omega\|_0\leq  \frac{\sqrt{\binom{n}{k}}}{1-c_1}\|\omega\|_{\mathcal{T}^{(r)}},\quad \forall \omega \in\forms{r}{k}(E).$$
\end{proposition}

\begin{figure}[H]
\caption{A visual representation of the construction of Proposition \ref{prop:convmesh} in the case $n=2$, $k=1$. The thick line represents the boundary of the convex fat set $E\subset\R^2$. Gray squares are the squares $Q_i$'s of the considered coordinate tessellation that intersect $E$. In this example the points $\tilde x_i$ and the two columns of the matrices $A^{(s,r)}$ are randomly chosen. Note that in such a way there is no control on the size of the support of the constructed functionals.} \label{fig:cassinimesh}
\begin{center}
\includegraphics[scale=0.9]{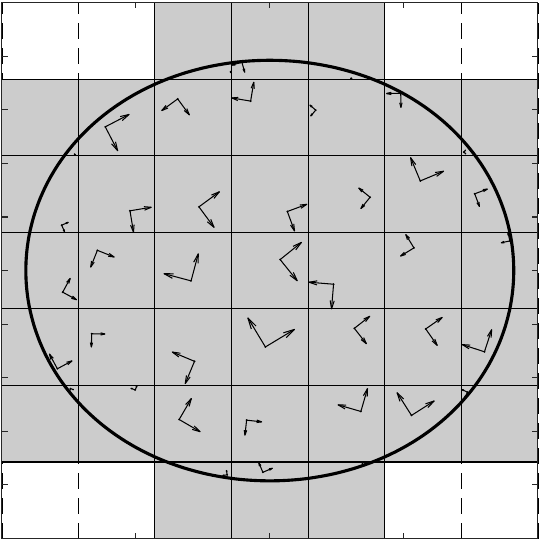}
\end{center}
\end{figure}

\begin{example}[An admissible integral $2$-mesh for the unit cube]\label{ex:markovsquare}
We can construct an admissible integral $2$-mesh for the unit cube $E:=[0,1]^3$ by slightly sharpening the construction of Proposition \ref{prop:convmesh}. 

We consider a tessellation of $E$ made of coordinate cubes with side 
$$d= \frac{c_1w(E)}{2\sqrt{n} \binom{n}{k}r^2}=\frac{c_1}{2\sqrt{n} \binom{n}{k}r^2}\,,$$
where $c_1<1$ is such that $d^{-1}\in \N$. For any  $x\in E$ we can find a closed cube $Q$ of the tessellation for which $x\in Q$. Let $v$ be the closest vertex of $Q$ to $x$ and let $\mathcal F^{(1)},\mathcal F^{(2)}, \mathcal F^{(3)}$ be the faces of $Q$ containing $v$. It is clear that the $2$-vectors $\tau^{(1)}, \tau^{(2)}, \tau^{(3)}$ representing any orientation of $\mathcal F^{(1)},\mathcal F^{(2)}, \mathcal F^{(3)}$ form an orthonormal basis of $\Lambda_k$ (i.e., property b) above still holds). It is also easy to see that
$$ \max_{j=1,\dots,\binom{n}{k}}\max_{z\in \mathcal F^{(s_j)}}|x-z|\leq \sqrt{n} d/2.$$  
Denote by $\mathcal T^{(r)}$ the set of integral averaging currents supported on all faces of the tessellation. By Theorem \ref{FE-CM}, we have
$$\|\omega\|_0\leq  \frac{\sqrt{\binom{n}{k}}}{1-c_1}\|\omega\|_{\mathcal{T}^{r}},\;\forall \omega \in\forms{r}{k}(E),$$
that is, $\mathcal T^{(r)}$ is an integral $2$-mesh of constant $\sqrt{\binom{n}{k}}/(1-c_1)$. The cardinality of this mesh is given by the number of faces in the tessellation, and can be computed. Consider an $n$-cube, and split each side in $ N $ parts. Then, the number of $ k $-cubes (i.e., the number of $ k $-dimensional faces of the so obtained complex) is $ \binom{n}{k} N^k (N+1)^{n-k} $. In our framework, this gives that the number of $2$-faces in the $3$-cube equals to
$$\Card \mathcal T^{(r)}=\binom{3}{2}\frac 1 {d^k}\left(\frac{1}{d}+1\right)^{n-k}=\mathcal O(r^6).$$
%\fede{Qui c'è qualche pesante casino di calcolo combinatorio.... METTERE A POSTO!!! Poi inserire versione giusta più avanti quando si fa il confronto con Baran.}
\end{example}

\begin{figure}[H]
\caption{Cardinalities of $\forms{r}{k}$-admissible integral $k$-meshes for the square (i.e., $n=2$) and the cube (i.e., $n=3$) constructed following Example \ref{ex:markovsquare} with $r$ varying from $1$ to $20$, and $k=1,\dots,n$, and $c_1=1/2$.}
\label{fig:markovsquare}
\begin{center}
\includegraphics[scale=0.6]{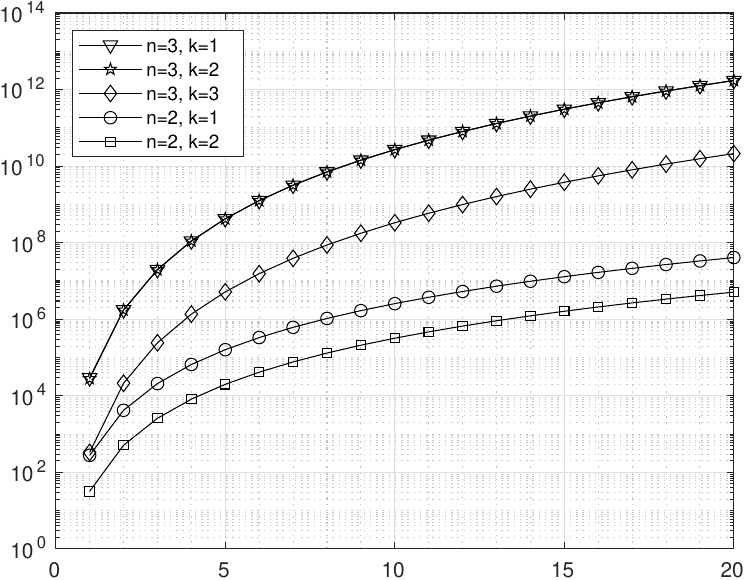}
\end{center}
\end{figure}

\begin{remark} \label{rmk:largemesh}
The size of meshes constructed by the technology based on Theorem \ref{FE-CM} grows very quickly, see Figure \ref{fig:markovsquare}. Although this strategy is very flexible and can be easily adapted to many scenarios, meshes of such a large cardinality are not of practical use for most applications. This motivates the development of an alternative technique, which is carried in the next section.
\end{remark}

\subsection{Constructing $\forms{r}{k}$-\mesh by Baran inequality} \label{sect:baran}
The classical Bernstein Inequality on the interval $E:=[-1,1]$ states that, for any polynomial $p$ such that $|p(x)|\leq 1$ for any $x\in E$, one has
\begin{equation}\label{BernsteinInequality}
\frac{|p'(x)|}{\sqrt{\max_{y\in E}|p(y)|^2-|p(x)|^2}}\leq \deg(p) \frac 1{\sqrt{1-x^2}},\quad \forall x\in\interior E.
\end{equation}
This inequalilty has immediate applications in polynomial sampling. %Indeed it is not difficult to prove that it implies
In particular, \eqref{BernsteinInequality} implies that for any polynomial $ p\in \mathscr P $ such that $ \|p\|_E\leq 1$, one has 
% \begin{equation}
% |\arcos(p(x))-\arcos(p(y))|\leq& \deg(p) \left|\int_x^y\frac{1}{\sqrt{1-s^2}}ds\right|\notag\\
% =&\deg(p)|\arcos(y)-\arcos(x)|\,,\;\;\forall p\in \mathscr P: \|p\|_E\leq 1.\label{samplingineqbyBern}
% \end{equation}
\begin{equation}\label{samplingineqbyBern}
|\arcos(p(x))-\arcos(p(y))|\leq \deg(p) \left|\int_x^y\frac{1}{\sqrt{1-s^2}}ds\right| = \deg(p)|\arcos(y)-\arcos(x)| .
\end{equation}
%In this context it is rather natural to 
\begin{example}[Chebyshev-Lobatto points as an admissible mesh]
Consider $[-1,1]$ as (the closure of) a Riemannian manifold isometric to the upper semicirle $\mathbb S^1_+:=\{x\in \R^2:|x|=1,x_2>0 \}$ endowed by the round metric $g$ that, once pulled back onto the open interval $(-1,1)$, induces the distance $d_E(x,y):=|\arcos(y)-\arcos(x)|$. Assuming that $p(x)=\|p\|_E=1$ for some $x\in \interior E$, it is possible to rewrite \eqref{samplingineqbyBern} as
\begin{equation*}
|\arcos p(y)|\leq \deg(p)d_E(x,y)\,, \quad \forall p\in \mathscr P.
\end{equation*}
Hence, if $h=d_E(x,y)<\pi/(2\deg(p))$ we have $|p(y)|\geq \cos(\deg(p) h).$ 
Therefore, if $\mathcal{T}=\{x_1,x_2,\dots,x_M\}$ is a set of points in $E$ such that
$$h(\mathcal{T},K):=\max_{y\in E}\min_{x\in \mathcal{T}}d_E(x,y)<\pi/(2r),$$
then one has the sampling inequality 
\begin{equation}\label{SIbybaran}
\max_{x\in E}|p(x)|\leq \frac 1{\cos(rh(\mathcal{T},K))}\max_{x\in \mathcal{T}}|p(x)|
\end{equation}
for any $p\in \mathscr P_r.$ In particular a set $\mathcal T_m$ of $m+1$ Chebyshev-Lobatto points %(i.e., $M=m+1$ and 
$x_i:=\cos(i\pi/(m+1))$ has the property $h(\mathcal T_m,[-1,1])=\frac \pi{2 m}$. It follows that $\max_{x\in[-1,1]}|p(x)|\leq \frac 1{\cos(\pi r/(2m)) }\max_{0\leq i\leq m}|p(x_i)|$ for any $m=m(r)>r.$ In other words $\{\mathcal T_{m(r)}\}$ is an admissible mesh for $E.$ 
 \end{example}
 
The generalization of the Bernstein Inequality \eqref{BernsteinInequality} to the case of several variables has been carried out by Baran \cite{Miro92,Miro94}. The crucial step in this generalization consists in identifying the function $1/\sqrt{1-x^2}$ appearing in \eqref{BernsteinInequality} as the normal derivative at $x$ of the Green function of $\C\setminus[-1,1]$ with a logarithmic pole at infinity. This quantity has a natural counterpart in several complex variables that, for a polynomially determining set $E \subset \C^n$, is the pluricomplex Green function%. If $E \subset \C^n$ is a compact polynomially determining set $E \subset \C^n$, is the Pluricomplex Green function %of $E$ is customarily denoted by $V_E^*$, and one has $V_E^*:\C^n\rightarrow [0,+\infty[$,
$$V_E^*(z):=\limsup_{\zeta\to z}\sup\left\{\frac 1{\deg p}\log|p(\zeta)|, p\in \mathscr P,\, \max_K|p|\leq 1\right\} ,$$
which is a maximal plurisubharmonic function solving the complex Monge Ampere homogeneous equation \cite{Norm12}. The generalization of \eqref{BernsteinInequality} then assumes the following form. % (see, e.g., \cite{Norm12} for a compact account on pluripotential theory and its applications in approximation theory).
\begin{theorem}[Baran Inequality \cite{Miro92}]\label{thm:baran}
Let $E \subset \R^n$ be a compact set with non empty interior. %, then we have
Then
\begin{equation}
\frac{|\partial_v p(x)|}{\sqrt{1-|p(x)|^2}}\leq \deg(p) \partial_v^+ V_E^*(x), \quad \forall p\in\mathscr P:\,\max_K|p|\leq 1,\,v\in \R^n,
\end{equation}
where 
$$\partial_v^+ V_E^*(x):=\liminf_{t\to 0^+}V_E^*(x+itv).$$
\end{theorem}

One can think of $F(x,v):=\partial_v^+ V_E^*(x)$ as a metric on the tangent space to $\interior E$ at $ x $. In general this is only a Finsler metric \cite{LenNormShayne04}, but in few very relevant situations it turns out to be a Riemannian metric \cite{Piazzon19torus}. We term such a metric the \emph{Baran metric} of $E$, and we denote it by $\delta_E(x,v)$. Via the Carnot-Caratheodory construction, $\delta_E(x,v)$ induces the distance
\begin{equation*}
d_E(x,y):=\inf\left\{\int_0^1 \delta_E(\gamma(t),\gamma'(t))dt, \gamma(0)=x, \ \gamma(1)=1, \ \support \gamma\subset K\right\},
\end{equation*}
which we call \emph{Baran distance} on $E$. Within this formalism, it is clear that equation \eqref{SIbybaran} holds in the multivariate context, provided that necessary changes are made. 

\begin{theorem}[Fundamental estimate by Baran Inequality]\label{th:FEB}
Let $E$ be a real body in $\R^n$. Let $\mathcal T=\{T^{(s)}\}_{s=1,\dots,M}\subset \averaging(E)$ be such that, for any $s=1,2,\dots, M$, there exist $x^{(s)}\in E$, $A^{(s)}\in M_{n,k}(\R^n)$ with $\Rank(A^{(s)})=k$, a compact set $\Omega^{(s)}\subset\R^k$ of positive finite Lebesgue measure, for which
$$T^{(s)}(\omega)= \frac 1{\haus^k(S^{(s)})}\int_{S^{(s)}}\langle\omega;\sigma^{(s)}\rangle \de\haus^k\;,\;\;\forall \omega\in \testforms(E)\,,$$
where $S^{(s)}$ is the $k$-rectifiable set $\{x^{(s)}+A^{(s)}y,y\in \Omega^{(s)}\}\subset E$ endowed by the orientation
$$\sigma^{(s)}:=\frac{A_{:,1}^{(s)}\wedge\dots\wedge A_{:,k}^{(s)}}{|A_{:,1}^{(s)}\wedge\dots\wedge A_{:,k}^{(s,r)}|}\,.$$ Let $r\in \N$ and assume that there exists $c_1<1$ such that, for any $x\in E$, we can find $s_1,s_2,\dots,s_m\in \{1,\dots,M\}$ such that
\begin{align}
c:=&\sup_{\tau\in \Lambda_k}\Bigg\{\min_{a\in \R^m}\Big\{|a|_1: \sum_{j=1}^m a_j \sigma^{(s_j)}=\tau\Big\}, |\tau|=1, \tau\text{ simple}\Bigg\}<+\infty\label{FE-A1B}\\
h:=&\max_{j=1,\dots,m}\max_{z\in S^{s_j}}d_E(x,z)< \frac{\pi}{2r}\,.\label{FE-A2B}
\end{align}
Then, for any $\omega\in \mathscr P_r\Lambda^k$ we have
\begin{equation}\label{samplinginequalityBaran}
\|\omega\|_0\leq\frac{c}{\cos(rh)} \|\omega\|_{\mathcal{T}}.
\end{equation}
\end{theorem}
\begin{proof}
Let $\omega\in \forms{r}{k}$ be any polynomial form. We prove that \eqref{samplinginequalityBaran} holds for $\omega$ under the assumption that there exist $\bar x\in E$ and a simple $\bar\tau\in \Lambda_k$ such that 
\begin{equation}\label{achievemaxBaran}
\max_{x\in E}|\omega(x)|^*=\langle\omega(\bar x);\bar \tau\rangle,
\end{equation}
the general case will easily follow by approximation.

Let $m\in \N$, $s_1,s_2,\dots,s_m\in\{1,\dots,N\}$ be such that
\begin{equation*}\label{specifyass2Baran}
\max_{j=1,\dots,m}\max_{z\in S^{s_j}}|\bar x-z|\leq h,
\end{equation*}
and let  $\bar a\in \R^m$ be any vector realizing
\begin{equation}\label{specifyass1Baran}
\min_{a\in \R^m}\Big\{|a|_1: \sum_{j=1}^m a_j \sigma^{(s_j)}=\bar \tau\Big\}\leq c.
\end{equation}
Notice that $p(x):=\langle\omega(x);\bar \tau\rangle$ is a polynomial (whose degree does not exceed $r$) achiving its uniform norm on $E$ at the point $\bar x$. Exploiting \eqref{achievemaxBaran} and  \eqref{specifyass1Baran} we compute 
\begin{align*}
\max_{x\in E}|\omega(x)|^*=&\langle\omega(\bar x);\bar \tau\rangle=:p(\bar x)\leq \frac 1{\cos(rd_E(\bar x,x))}p(x)\leq \frac 1{\cos(rh)} \langle\omega(x);\bar \tau\rangle\\
=&\frac 1{\cos(rh)}\sum_{j=1}^m\bar a_j\langle \omega(x);\sigma^{(s_j)}\rangle\leq \frac 1{\cos(rh)}\sum_{j=1}^m\bar a_j\frac{1}{\haus^k (S^{s_j})}\int_{S^{s_j}}\langle \omega( x);\sigma^{(s_j)}\rangle \de\haus^k(x)\notag\\
\leq& \frac{|\bar a|_1}{\cos(rh)}\|\omega\|_{\mathcal{T}}\,.
\end{align*}
This shows the claim. \myqed
\end{proof}

\begin{example}[Optimal admissible integral $k$-mesh for the $n$-cube by Baran Inequality]\label{ex:baransquare}
Let $m>r$ and set $\tilde x_i:=\cos(i \pi/m)$ for any $i\in \{0,\dots,m\}.$ Consider the grid constructed by $n$-th Cartesian product of the set $\{\tilde x_0,\dots,\tilde x_m\}$ and the corresponding tessellation $\{Q_j\}_{j=1,\dots,m^n}$ of $E:=[-1,1]^n$ made of $n$-dimensional parallepipeds. Let $\mathcal F^{1,r},\dots,\mathcal F^{N,r}$, with $N=\binom{n}{k}m^k(m+1)^{n-k}$, be the collection of all the $k$-dimensional faces of the parallepipeds of the tessellation, with no repetitions. Any $k$-face $\mathcal F^{(s,r)}$ is of the form 
$$\mathcal F^{(s,r)}=\Big\{x\in E: x_{\alpha(j,s)}=\tilde x_{i(j,s)},\,j=1,\dots,n-k,\, \tilde x_{i(j,s)}\leq x_{\alpha(j,s)}\leq \tilde x_{i(j,s)+1},\,j=n-k+1,\dots,n\Big\}\,,$$
where $\alpha(\cdot,s)$ is a permutation of the set $\{1,2,\dots,n\}$ and $i(j,s)\in\{0,\dots,m-1\}$ for any $j=1,2,n$, and $s=1,\dots, N.$ The tangent space to such a $k$-face is then spanned by $e_{\alpha(n-k+1,s)},\dots,e_{\alpha(n,s)}.$ Therefore we can give a parametrization and an orientation to $\mathcal F^{(s,r)}$ by setting, for each $ j=1,\dots,n$ and $s=1,\dots,N$,
\begin{align*}
x^{(s,r)}_{\alpha(j,s)}=&\tilde x_{i(j,s)}\\
A^{(s,r)}=&[\pm e_{\alpha(n-k+1,s)},\dots,\pm e_{\alpha(n,s)}]\\
\Omega^{(s,r)}=&\otimes_{j=n-k+1}^n[0,\tilde x_{i(j,s)+1}-\tilde x_{i(j,s)}]\,,
\end{align*}
where signs are suitably chosen in order to have $\mathcal F^{(s,r)}=\{x^{(s,r)}+A^{(s,r)}y,\,y\in \Omega^{(s,r)}\}\subset E.$

Now, recalling that for the $n$ dimensional cube $E$ one has 
$$d_E(x,y):=\max_{i\in\{1,\dots,n\}}|\arcos(x_i)-\arcos(y_i)|,$$
it is immediate to verify that $\mathcal T^{(r)}=\{(x^{(s,r)},A^{(s,r)},\Omega^{(s,r)}), s=1,\dots,N\}$ satisfies Eq. \eqref{FE-A1B} with $c=\sqrt{\binom{n}{k}}$ and Eq. \eqref{FE-A2B} with $h=\pi/(2m)$. Hence, due to Theorem \ref{th:FEB}, $\{\mathcal T^{(r)}\}$ is an \mesh of constant $C=\sqrt{\binom{n}{k}}/\cos(\frac{\pi r}{2m}).$ 
\end{example}
\begin{remark}
In the nodal framework admissible polinomial meshes $\{\mathcal T^{(r)}\}_{r\in \N}$  such that $\Card \mathcal T^{(r)}\sim \ddim \mathscr P_r$ are termed \emph{optimal}, see \cite{Kroo11,Piazzon16,Kroo19,DaPr24}. It is worth noticing that the same property holds for the integral admissible mesh constructed in Example \ref{ex:baransquare} above; this also justifies its title.
\end{remark}

\begin{remark}\label{rmk:comparecardinality}
In order to fairly compare the cardinality of the mesh constructed in Example \ref{ex:baransquare} with that constructed in Example \ref{ex:markovsquare} (reported in Figure \ref{fig:markovsquare} for $c_1=1/2$), we need to set $h$ as the same constant of the mesh, i.e.,
$$C=\frac{\sqrt{\binom{n}{k}}}{1-c_1}=\frac{\sqrt{\binom{n}{k}}}{\cos(\frac{\pi r}{2m})},$$
that leads to $m=\lceil3r/2\rceil$. For such a choice we have $M(r)\sim\binom{n}{k}\frac{3^k}{2^n}r^k(3r+2)^{n-k}$. The comparison of the cardinality of the two families, reported in in Figure \ref{fig:comparecard} for the case $n=3$, $k=2$, shows that this latter strategy yields sensibly smaller sets.
\end{remark}

\begin{figure}[H]
\caption{Comparison of the cardinality of the mesh constructed in Example \ref{ex:baransquare} with the one constructed in Example \ref{ex:markovsquare} in the case of $c_1=1/2$, $n=3$, and $k=2$. As a reference we depict also the curve $r\mapsto \ddim \forms{r}{k}.$ }\label{fig:comparecard}
\begin{center}
\includegraphics[scale=0.6]{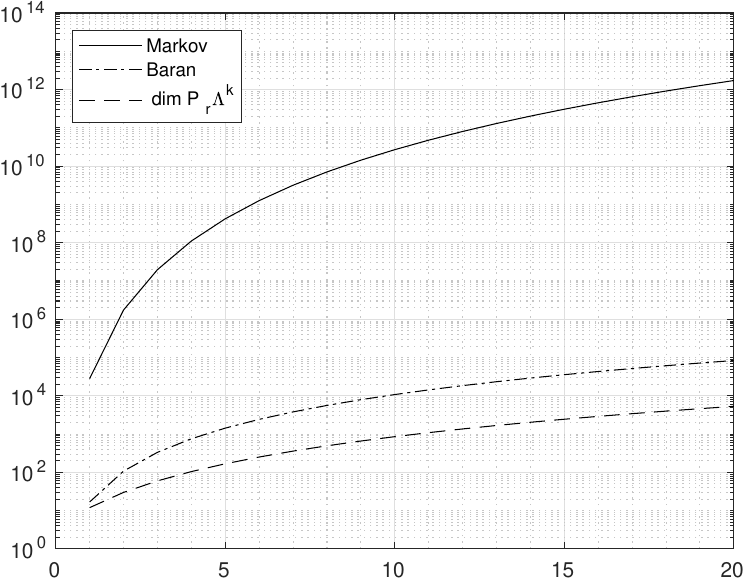}
\end{center}
\end{figure}

\subsection{Admissible integral $k$-meshes for the $2$-simplex and the $3$-simplex by Baran Inequality}\label{subsec:simplex}
We dedicate this subsection to the explicit construction of a low cardinality \mesh for the $n$-simplex, with $n=2$ or $3$, relying upon Theorem \ref{th:FEB}. This algorithmic construction can be likewise extended to the case $ n > 3 $, but requires a significantly heavier notation.

%We will construct a integral $k$-mesh for the $n$-simplex by means of $k$-faces of families of coordinate simplices, i.e. sets of the form
To this end, we consider $k$-faces of families of coordinate simplices, i.e. sets of the form
\begin{equation*}\label{eq:coordinatesimplex}
S(\tau,a,\ell):=\Big\{x\in \R^n:\,\tau_i(x_i-a_i)\geq 0,\,\sum_{i=1}^n \tau_i(x_i-a_i)\leq \ell\Big\},
\end{equation*}
%We introduce the following notation for the $k$-integral meshes we are going to use
and propose the following notation for the $k$-integral meshes we construct.
\begin{definition}\label{def:mymesh}
Let $\tilde {\mathcal T}_{k,m,n}$ denote
\begin{itemize}
\item[\underline{\textbf{For} $\boldsymbol{n=2}$}] 
\begin{itemize}
\item for $k=1$, the collection of edges of simplices $S(\tau,a,\ell)$, with $\tau=\{1,1\}$, $\ell=1/m$, and $a\in((m-1)\Sigma_n\cap \N^2)=\{(i,j)\in \N^2: i+j\leq m-1\}$
\item for $k=2$,  the collection of $2$-simplices of the form $S(\tau,a,\ell)$, with either $\tau=\{1,1\}$, $\ell=1/m$, and $a/\ell\in((m-1)\Sigma_n\cap \N^2),$ or $\tau=\{-1,-1\}$, $\ell=1/m$, and $(a+\tau)/\ell\in (m-1)\Sigma_n\cap \N^2.$
\end{itemize} 
\item[\underline{\textbf{For} $\boldsymbol{n=3}$}] \begin{itemize}
\item for $k<n$, the collection of $k$-faces of simplices $S(\tau,a,\ell)$, with $\tau=\{(-1)^{s_1},(-1)^{s_2},(-1)^{s_3}\}$, $(s_1+s_2+s_3)\equiv 0 (\mathrm{mod} 2)$, $\ell=1/m$, and $a/\ell\in (m-1)\Sigma_n\cap \N^3$
\item for $k=n$ the collection of the $n$-simplices introduced in the case $k<n$ and the collection of the tetrahedrons $T(a,\ell)$ of verices $a+(\ell,0,0)$, $a+(0,\ell,0)$, $a+(0,0,\ell)$,$a+(\ell,\ell,\ell)$, with $\ell=1/m$ and $a/\ell\in (m-1)\Sigma_n\cap \N^3$.
\end{itemize}
\end{itemize}
\end{definition}
Figure \ref{fig:elements} illustrates the elements of the tessellation proposed in Definition \ref{def:mymesh} for the case $n=3$.

\begin{figure}[h]
\caption{Elements used in Definition \ref{def:mymesh} in the case $n=3$. \emph{Left:} the four type of simplices that are always (i.e., $k=1,2,3$) considered. \emph{Right:} the tetrahedron that is taken into account only in the case $k=3$.} \label{fig:elements}
\begin{center}
\begin{tabular}{cc}
\includegraphics[width=0.49\textwidth]{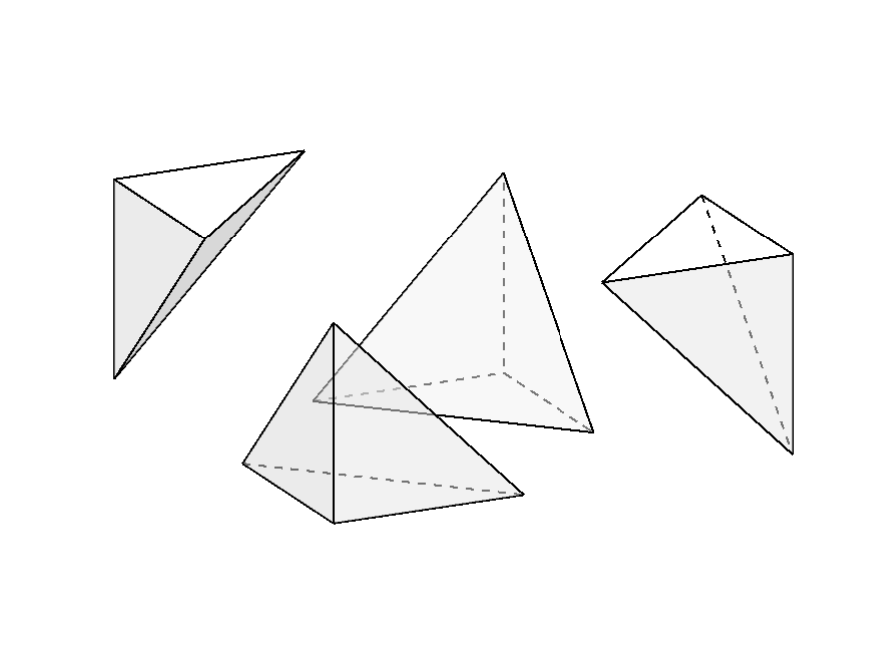}&
\includegraphics[width=0.39\textwidth]{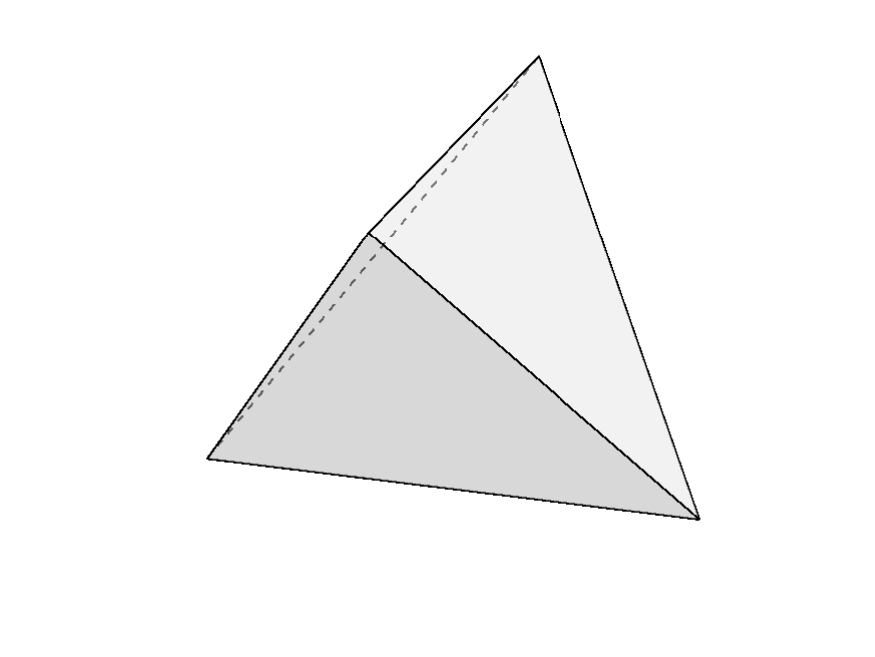}
\end{tabular}
\end{center}
\end{figure}

As shown in \cite{Miro92,LenNormShayne04}, the Baran distance on the standard $n$-dimensional simplex $\Sigma_n:=\{x\in \R^n: x_i\geq 0,\,\sum_{i=1}^nx_i\leq 1\}$ has the explicit form
\begin{equation*}
d_{\Sigma_n}(x,y)=2 \arcos\left(\phi(x),\phi(y)  \right),
\end{equation*}
where $\phi:\Sigma_n\rightarrow \mathbb S_{n}\cap \R_+^{n+1}$ is defined as
$$\phi(x_1,\dots,x_n)=\left(\sqrt{1-\sum_{i=1}^n x_i},\sqrt{x_1},\dots,\sqrt{x_n}\right) $$
and $d_{\Sigma_n}$ is the Carnot-Caratheodory distance induced by the Riemannian metric 
$$g_{\Sigma_n}(x):=\left[\begin{array}{ccccc}
x_1^{-1}     & 0&\dots& \dots &0\\
0& x_2^{-1}&0& \dots &0\\
\vdots  & \vdots    &\vdots&\vdots&\vdots\\
\vdots  & \vdots    &\vdots&\vdots&\vdots\\
0     & \dots&\dots& 0 &x_n^{-1}
\end{array}\right]+
\frac 1{1-\sum_{i=1}^n x_i}
\left[\begin{array}{ccccc}
1& 1&\dots& \dots &1\\
1& 1&\dots& \dots &1\\
1& 1&\dots& \dots &1\\
1& 1&\dots& \dots &1\\
1& 1&\dots& \dots &1
\end{array}\right].$$
defined on the interior of $\Sigma_n.$

\begin{remark} \label{rmk:volblowingup}
    The volume form $d\Vol_{\Sigma_n}=\frac 1{\sqrt{(1-\sum_{i=1}^nx_i)\prod_{i=1}^n x_i}}\de\haus^n(x)$ induced by $g_{\Sigma_n}$ blows up around the boundary of $\Sigma_n$. As a consequence, two points with fixed Euclidean distance present a large Baran distance as they approach the boundary of $\Sigma_n$. %The algorithm we develop relies upon an adaptive strategy justified by the following elementary observation
\end{remark}
Since we aim at constructing an integral $k$-mesh with maximal diameter of the elements of the tessellation bounded above by $h>0$, this justifies the following adaptive strategy for the development of Algorithm \ref{alg:meshthesimplex}:
\begin{itemize}
\item We compute a suitable $m$ such that the maximal diameter of elements of the tessellation defining  $\tilde {\mathcal T}_{k,m,n}$ is bounded above by $h$;
\item We add to our integral $k$-mesh all the $k$-faces of elements $E_\beta$ of this tessellation such that $E_\beta\cap \partial \Sigma_n\neq \emptyset$, i.e., elements with larger Baran diameter;
\item We iteratively repeat this procedure on the set until $\Sigma_n$ has been exhausted. Notice that in fact $\Sigma_n \setminus \cup_\beta E_\beta$ is still a simplex.
\end{itemize}
Such a procedure is formalized in Algorithm \ref{alg:meshthesimplex}, and graphically depticted in Figure \ref{fig:meshthesimplex}.
\begin{algorithm}[h!]
\caption{Simplex $k$-Mesh}\label{alg:meshthesimplex}
\begin{algorithmic}[1]
\Require $n\in\N,$ $k\in\{0,n\}$, $r\in \N$, $h\in]0,\pi/(2 r)[$
\State $c=0$, $s=1$, $T=\emptyset$, $p=1-\cos(h/2)$
\While{$s>0$}
\State $m \gets \left\lceil\frac{s}{p\left(2\sqrt{c/p}+1\right)}\right\rceil$
\State $\ell\gets s/m$
\State $\tilde T \gets c+s\tilde T_{k,m,n}$
\State $c\gets c+\ell$
\State $s\gets s -(n+1)\ell$
	\For{$E\in\tilde T$}
		\If{$E\not\subseteq c+s\Sigma_n$}
			\For{$F\in\{k\text{-faces of } E\}$}
				\State $T \gets T\cup F$
			\EndFor
		\EndIf
	\EndFor
\EndWhile
\Ensure $T$
\end{algorithmic}
\end{algorithm}
%
%The rest of the present section concern the proof of the following:
\begin{proposition}\label{prop:provealg}
Let $\mathcal T^{(r)}$ be the sequence of outputs of Algorithm \ref{alg:meshthesimplex} executed for each $r\in \N$ with $h:=\theta \pi/(2r)$ with $0<\theta<1$. Then $\mathcal T^{(r)}$ is an \mesh for the simplex with constant $C:=\frac {\sqrt{\binom{n}{k}}}{\cos(\theta\pi/2)}.$ 
\end{proposition}

Proposition \ref{prop:provealg} ensures that Algorithm \ref{alg:meshthesimplex} produces the desired admissible integral $k$-mesh. We present the steps of the proof as separate lemmata, and postpone the rest of the proof to the end of the section.

\begin{figure}[h!]
\caption{The costruction carried out by Algorithm \ref{alg:meshthesimplex}. At each step the elements lyining in the grey area are not taken into account, in the next step such region is re-meshed by simplices having sides of updated length.}\label{fig:meshthesimplex}
\begin{center}
\begin{tabular}{ccc}
\includegraphics[width=0.3\textwidth]{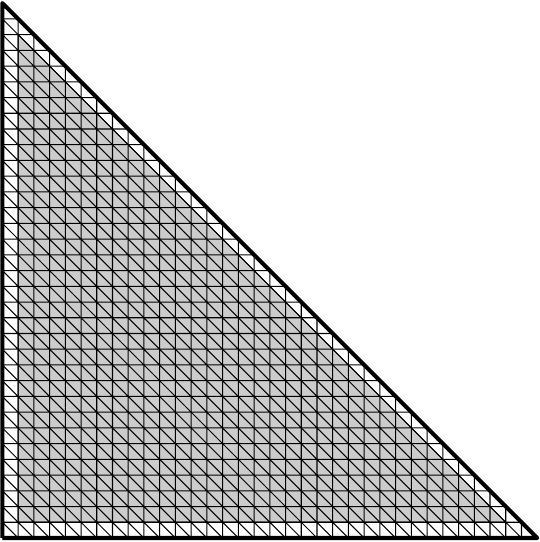}&
\includegraphics[width=0.3\textwidth]{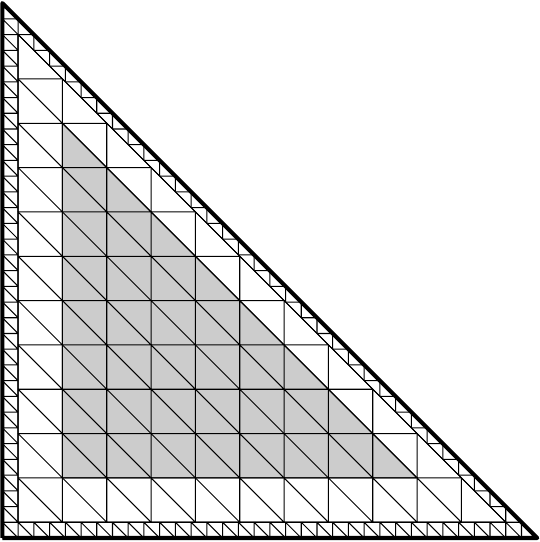}&
\includegraphics[width=0.3\textwidth]{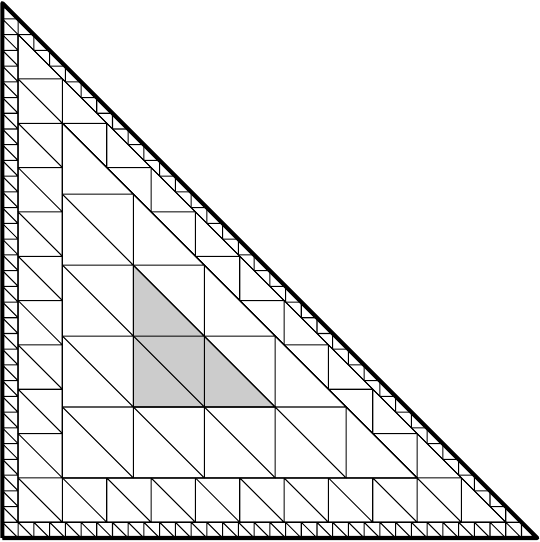}
\end{tabular}
\end{center}
\end{figure}

The first step we need to accomplish in order to prove Proposition \ref{prop:provealg} is to compute, for any $k$ and $r$, a suitable $m$ for which we can apply Theorem \ref{th:FEB} in which estimates \eqref{FE-A1B} and \eqref{FE-A2B} are written for the averaging currents defined by the $k$-faces of the element of the considered tessellation $\mathcal T_{k,m,n}$ that contains the test point $x\in \Sigma_n$. The following property of the Baran distance on the simplex will play a pivotal role in our construction.

\begin{lemma}\label{lemmadiam}
Let $f:\Sigma_n\times \Sigma_n\rightarrow R$ be defined as $f(x,y)=\cos( d_{\Sigma_n}(x,y)/2)$. Then the functions $f(\cdot,y)$ and $f(x,\cdot)$ are concave on $\Sigma_n$. In particular the Baran diameter of any convex polytope $\mathcal P\subset \Sigma_n$ is the maximum Baran distance of its vertices.
\end{lemma}

\begin{proof}
We show that $x\mapsto f(x,y)$ is convcave and the concavity of $y\mapsto f(x,y)$ will follow by simmetry. Indeed we can compute
$$\frac{\partial^2f}{\partial x_i\partial x_j}(x,y)=-\frac 1 4\left(\delta_{i,j}\sqrt{\frac{y_i}{x_i^{3}}}+ \frac{\sqrt{1-\sum_{h=1}^ny_h}}{\sqrt{1-\sum_{h=1}^nx_h}}  \right).$$
Therefore
\begin{align*}
\sum_{i,j=1}^n v_iv_j\frac{\partial^2f}{\partial x_i\partial x_j}(x,y)=&-\frac 1 4\left(\sum_{i=1}^n\sqrt{\frac{y_i}{x_i^{3}}}v_i^2+ \frac{\sqrt{1-\sum_{h=1}^ny_h}}{\sqrt{1-\sum_{h=1}^nx_h}}\sum_{i=1}^n\sum_{j=1}^n v_iv_j\right) \\
=&-\frac 1 4\left(\sum_{i=1}^n\sqrt{\frac{y_i}{x_i^{3}}}v_i^2+ \frac{\sqrt{1-\sum_{h=1}^ny_h}}{\sqrt{1-\sum_{h=1}^nx_h}}(\sum_{i=1}^n v_i)^2\right)\leq 0.
\end{align*}

Let us denote by $\diam_{\Sigma_n}$ the diameter with respect to the Baran distance. Notice that
\begin{align*}
\frac{\diam_{\Sigma_n}(\mathcal P)}2=&\max_{x\in \mathcal P} \max_{y\in \mathcal P}\arcos\left(\phi(x),\phi(y)  \right)=\max_{x\in \mathcal P} \arcos\min_{y\in \mathcal P}\left(\phi(x),\phi(y)  \right)=\arcos\min_{x\in \mathcal P} \min_{y\in \mathcal P}\left(\phi(x),\phi(y)  \right)\\
=& \arcos\min_{x\in \mathcal P} \min_{y\in \mathcal P}f(x,y)=\arcos\min_{x\in\mathcal P} \min_{y\in \extr \mathcal P}f(x,y)=\arcos\min_{x\in \extr \mathcal P} \min_{y\in \extr \mathcal P}f(x,y),
\end{align*}
where we denoted by $\extr \mathcal P$ the set of extremal points (hence vertices) of the convex polytope $\mathcal P$, which is indeed the set of vertices. Note that the last two equalities follow by the above proven concavity of $x\mapsto f(x,y)$ and from the concavity of $y\mapsto \min_{x\in \extr\mathcal P}f(x,y).$ \myqed
%Since $f(x,\cdot)$ is concave, its minimum on $S$, which is $n$-dimensional convex polytope, is attained on the boundary of $S$. Also, $\min_{y\in S}f(\cdot,y)$ is concave, being the point-wise minimum of a family of concave functions. Hence the minimum of $\min_{y\in S}f(\cdot,y)$ is also attained on the boundary of $S$. Therefor there exists at least a pair $(x^*,y^*)\in \partial S\times \partial S$ such that  $\frac{\diam_{\Sigma_n}(S)}2= \arcos\left(\phi(x^*),\phi(y^*)  \right).$
%
%Now observe that $\partial S=\cup_{j=1}^{n+1}S_j$, where the $S_j$'s are $n-1$ simplices. Hence we can iterate the above reasoning $n$ times to obtain that $x^*$ and $y^*$ belong to the $0$-dimensional part of the boundary ($0$-skeleton \cite{}) of $S$, i.e., the set of vertices of $S$.  
\end{proof}

We then give an asymptotic upper bound to the Baran diameter of the elements of a tessellation.

\begin{lemma}
Let $n=2$ or $n=3$ and let $\mathcal A$ be the set of the elements of the tessellation $\tilde{\mathcal T}(k,m,n)$ described in Definition \ref{def:mymesh}. Then
\begin{equation}\label{diambound}
\max_{A\in \mathcal A}{\diam}_{\Sigma_n}(A)=2\arcos(1-1/m)\,.
\end{equation} 
\end{lemma}
\begin{proof}
First, notice that, for the case of tessellations containing also tetrahedrons, we do not need to take them into account. Indeed, due to Lemma \ref{lemmadiam}, if $\mathcal P$ is a tetrahedron, its diameter is the distance of its most separated vertices, but, for any pair $(v_1,v_2)$ of such vertices, there exists a simplex $S$ of the same tessellation such that $u$ and $v$ are vertices of $S$. Hence, removing $\mathcal P$ from the maximization in \eqref{diambound} does not change the value of the maximum diameter.

Now, consider the elements $A\in \mathcal A(\tau)$ being simplices with the same $\tau$. Then we have
\begin{align*}
&\max_{A\in \mathcal A(\tau)}{\diam}_{\Sigma_n}(A)=\max_{m a\in(m-1)\Sigma_n\cap \N^n }{\diam}_{\Sigma_n}S(\tau,a,1/m)\\
=&2\arcos\big(\min_{m a\in(m-1)\Sigma_n\cap \N^n }\min_{x,y\in S(\tau,a,1/m)}f(x,y)\big)\\
=&2\max_{a\in\extr[((m-1)\Sigma_n\cap \N^n)/m] } \arcos\big(\min_{x,y\in S(\tau,a,1/m)}f(x,y)\big)\\
=&2\max_{a\in\extr[((m-1)\Sigma_n\cap \N^n)/m] } \arcos\big(\min_{x,y\in\extr S(\tau,a,1/m)}f(x,y)\big),
\end{align*}
where in the last two lines we used the fact that $a\mapsto \min_{x,y\in S(\tau,a,1/m)}f(x,y)=\min_{x,y\in S(\tau,0,1/m)}f(x+a,y+a)$ is concave, which is a straightforward consequence of Lemma \ref{lemmadiam}.

Then equation \eqref{diambound} is obtained by direct computation for each $\tau$ and using the simmetry of the problem. In order to clarify the procedure we report only the computation for the case $\tau=(1,\dots,1)$, $a=0$, since the other cases are completely analogous. In such a case one has
\begin{align*}
&\min_{x,y\in\extr S(\{1,\dots,\},0,1/m)}f(x,y)=\min\left\{\min\{f(0,e_i/m)\},\min\{f(e_i/m,e_j/m),i\neq j\}  \right\}\\
=&\min\left\{
\left\|
\begin{pmatrix}
\sqrt{1-1/m}\\
0\\
\vdots\\
0\\
\sqrt{1/m}\\
0\\
\vdots
\end{pmatrix}
\right\|^2\,,\,
\left(
\begin{pmatrix}
\sqrt{1-1/m}\\
0\\
\vdots\\
\sqrt{1/m}\\
0\\
0\\
\vdots
\end{pmatrix},
\begin{pmatrix}
\sqrt{1-1/m}\\
0\\
\vdots\\
0\\
\sqrt{1/m}\\
0\\
\vdots
\end{pmatrix}
\right)
\right\}\\
=&\min\{1,(1-1/m)\}=1-1/m\,,
\end{align*} 
and the claim is proved. \myqed
\end{proof}

Completely analogous computations can be used to bound the maximum Baran diameter of the elements of the similarly defined tessellation of certain subsimplices, giving the following result.% that we are going to use below.
\begin{corollary}
Let $0<c<1/(n+1)$, $n=2$ or $n=3$, and let $\mathcal A$ be the set of the elements of the tessellation $(1-c(n+1))\tilde{\mathcal T}(k,m,n)$ of $(1-c(n+1))\Sigma_n$, with $\tilde{\mathcal T}(k,m,n)$  described in Definition \ref{def:mymesh}. Then
\begin{equation}\label{diambound2}
\max_{A\in \mathcal A}{\diam}_{\Sigma_n}(A)=2\arcos\left( 1-\left(\sqrt{c+\frac{s}{m}}-\sqrt{c}\right)^2   \right)\,,
\end{equation}
where $s=1-(n+1)c.$ 
\end{corollary}

\begin{proof}[of Proposition \ref{prop:provealg}]
Pick any $r\in \N$ and $h<\frac{\pi}{2r}$. It is clear by construction and by Eq. \eqref{diambound} and Eq. \eqref{diambound2} that Algorithm \ref{alg:meshthesimplex} will provide an integral $k$-mesh $\mathcal T^{(r)}$ such that for any $x\in \Sigma_n$, we can find either simplex or a tetrahedron $S$ with $\diam_{\Sigma_n}(S)\leq h$, such that all its $k$-faces are in $\mathcal T^{(r)}$, and $x\in S.$ Hence the condition \eqref{FE-A2B} of Theorem \ref{th:FEB} is satisfied.

If $S$ is a coordinate simplex, then it is clear that \eqref{FE-A1B} of Theorem \ref{th:FEB} is satisfied as well with $c=\sqrt{\binom{n}{k}}.$ If $S$ is one of the tetrahedrons described in Definition \ref{def:mymesh}, hence we are considering the case $k=n$, we can take $c=1$, since in such a case $\Lambda^k\cong\R$, all $k$-vectors are in particular simple and the integer $m$ in \eqref{FE-A1B}  is $1$.

Now we need to provide an asymptotic upper bound for the cardinality of $\mathcal T^{(r)}$ as $r\to +\infty$. Notice that the sequence of real numbers $\ell$ computed by Algorithm \ref{alg:meshthesimplex} is non decreasing. At the first iteration, Algorithm \ref{alg:meshthesimplex} starts from a tessellation of the simplex made of simplices of side $\ell_1=1/\lceil 1/p\rceil$, hence it considers 
$$M_1=\frac{\leb^n(\Sigma_n)}{\leb^n(\ell_1\Sigma_n)}=\ell_1^{-n}= \lceil 1/p\rceil^n$$
simplices. Then, at each iteration the sub-tessellation covering the simplex $(1-(n+1)\sum_{l=1}^{j-1}\ell_j)\Sigma_n$ is replaced by a tessellation made by simplices of larger side, i.e., the considered tessellation has smaller cardinality. Therefore we have
$$\Card \mathcal T^{(r)}\leq \binom{n+1}{k}M_1=\binom{n+1}{k} \lceil 1/p\rceil^n,$$
where $\binom{n+1}{k}$ is the number of $k$-faces of an $n$ simplex.
One finally has
$$\lceil 1/p\rceil^n\sim \left(\frac 1{1-\cos \frac h 2}\right)^n=\left(\frac{1+\sqrt{\frac{1+\cos h}2}}{1-\frac 1 2-\frac{\cos h}2}\right)^n\sim \left(\frac 2 {h^2}\right)^n=\left(\frac{2}{\theta^2\pi^2}r^2\right),$$
as $r\to +\infty.$
Since the set $\mathcal T^{(r)}\subset \averaging(\Sigma_n)$ satisfies conditions \eqref{FE-A1B} and \eqref{FE-A2B} of Theorem \ref{th:FEB}, and has a cardinality that grows polynomially with respect to $r$, the sequence $\{\mathcal T^{(r)}\}$ is an \mesh for $\Sigma_n.$ \myqed
\end{proof}

\begin{remark}
We stress that the asymptotic bound for the cardinality of the output of Algorithm \ref{alg:meshthesimplex} provided in the proof of Proposition \ref{prop:provealg} is extremely rough. The corresponding integral $k$-mesh in fact presents a cardinality which increases proportionally to $ \log(r) \ddim \forms{r}{k}$, as depicted in Fig. \ref{fig:cardsimplex}. In the nodal case $k=0$, meshes satisfying such a cardinality growth rate are termed \emph{quasi-optimal}, while meshes with the cardinality growing as a multiple of $\ddim \forms{r}{k}$ are termed \emph{optimal}, see, e.g., \cite{DaPr24,Kroo11,Piazzon16}. 
\end{remark}
\begin{figure}
\caption{\emph{Left:} cardinalities $\Card \mathcal T^{(r)}$ of the \mesh constructed by Algorithm \ref{alg:meshthesimplex} (with $n=2$, $k=1$, $\theta=2/3$, so that $h=\theta\pi/(2 r)$ has the property $1/\cos(hr)=2$) compared with the growth of $r^2\log(r)$. \emph{Right:} the ratio $(\Card \mathcal T^{(r)})/ (r^2\log(r)).$}\label{fig:cardsimplex} 
\begin{center}
\begin{tabular}{cc}
\includegraphics[width=0.45\textwidth]{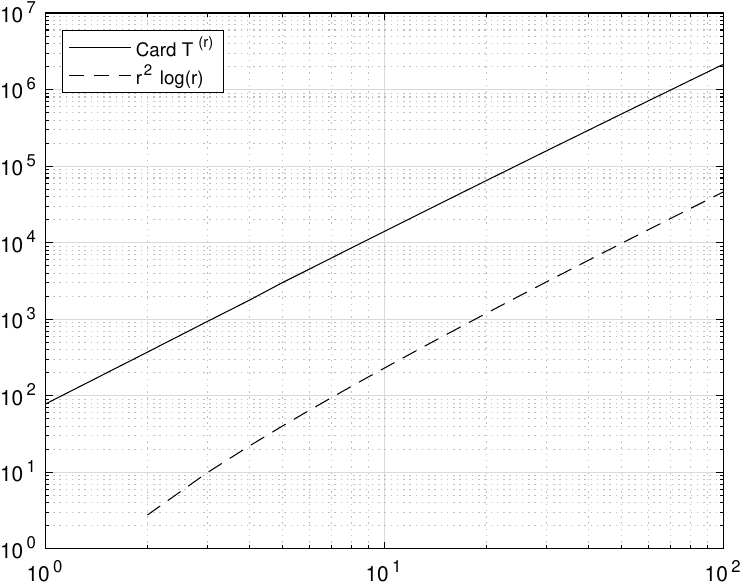}&
\includegraphics[width=0.45\textwidth]{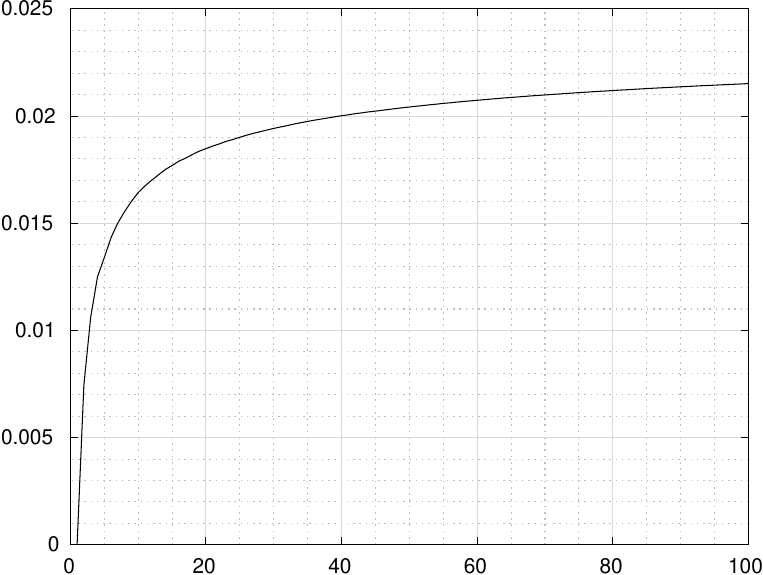}
\end{tabular}
\end{center}
\end{figure}

% We remark that the constructions of \mesh we propose in the present and the previous subsections rely on classical Markov or Bernstein type inequalities for polynomial \emph{functions} with respect to the uniform norm over $E$. This observation motivates the following open problem:
% \begin{openproblem}
% Can we state suitable extension of Markov and Bernstein inequalities to the framework of polynomial $k$-differential forms in a way that allows for the construction of \meshes?

  % construction of meshes
%\section{Interpolation and fitting on admissible integral $k$-meshes} \label{sec:5}

\section{Error estimates for least squares fitting on admissible integral $k$-meshes} \label{sec:5}

\subsection{Fitting of a set of currents}
Let $ \spaceV$ be a $N$-dimensional linear subspace of $ \testforms(E) $, $\mathcal T:=\{T_1,\dots, T_M\}\subset \currents(E)$. We term $\mathcal T$ \emph{determining} for $\spaceV$ whenever the condition $T(\omega)=0$, $\forall T\in \mathcal T$, implies $\omega=0$; note that necessarily we have $\Card \mathcal T\geq N$ in such a case. 
%
%The extremal case happens when $\mathcal T$ is determining and $\Card \mathcal T=N$, then $\mathcal T$ is termed \emph{unisolvent}, referring to the well posedness of interpolation, i.e. the fact that, for any $b\in \R^N$, we can find $\omega\in \spaceV$ such that $T_i(\omega)=b_i$ for $i=1,2,\dots,N$.
For any $\spaceV$-determining set $\mathcal T:=\{T_i\}_{i=1}^M\subset \averaging(E)$, and a positive vector of weights $w\in \R_{>0}^M$ we introduce the scalar product $(\cdot \, , \cdot)_{\mathcal T,w}$ on $\spaceV$ by setting
\begin{equation}\label{scalarproduct}
(\omega,\theta)_{\mathcal T,w}:=\sum_{i=1}^Mw_iT_i(\omega)T_i(\theta)\,.
\end{equation}
Note that $(\omega,\omega)_{\mathcal T,w}=\sum_{i=1}^M(\sqrt{w_i}T_i(\omega))^2=0$ implies $T_i(\omega)=0$ for all $i=1,\dots, M$ and thus $\omega=0$ since $\mathcal T$ is determining. Also, the right hand side of \eqref{scalarproduct} naturally extends to a non-negative simmetric bilinear form on $\testforms(E)$, thus $\|\omega\|_{\mathcal T,w}:=(\omega,\omega)_{\mathcal T,w}$ is a seminorm on $\testforms(E)$ (and in particular a norm on $\spaceV$). Therefore, the weighted discrete least squares projector $P_{\mathcal T,w}:\testforms(E)\rightarrow\mathscr{V} (E) $, 
\begin{equation}\label{discreteleastsquares}
P_{\mathcal T,w}\omega:= \argmin\{\|\omega-\theta\|_{\mathcal T,w}^2 ,\;\theta\in \mathscr{V} (E) \},\quad \forall \omega\in \testforms(E), 
\end{equation} 
is then well-defined due to Pythagorean Theorem. In particular, we can compute an orthonormal basis $\{\eta_1,\dots,\eta_N\}$ of $\mathscr{V} (E) $ and write
\begin{equation} \label{eq:samplingProjector}
P_{\mathcal T,w}\omega=\sum_{h=1}^N(\omega,\eta_h)_{\mathcal T,w}\eta_h.
\end{equation} 
Let us recall that the Lebesgue constant $\L$ of the fitting problem is defined by
\begin{equation}\label{Lebesgueconstant}
\L(\spaceV,\mathcal T,w):=\sup_{T\in \averaging(E)}\sum_{i=1}^{M}\left|\sum_{h=1}^NT_i(\eta_h)T(\eta_h)\right|w_i\,.
\end{equation}
where $\eta_h$'s are any $(\cdot,\cdot)_{\mathcal T,w}$-orthonormal basis of $\spaceV$. This quantity is a tight upper bound for the norm of the projection $P_{\mathcal T,w}$. Indeed we always have
\begin{equation*}
\opnorm{P_{\mathcal T,w}}:=\sup_{\omega\in \testforms(E)\setminus 0}\frac{\|P_{\mathcal T,w}\omega\|_0}{\|\omega\|_0}\leq \L(\spaceV,\mathcal T,w),
\end{equation*}
and equality holds when $\haus^k(\support T_i \cap \support T_j)=0$ for distinct indices $i$ and $j$, \cite{BrPi25}. Therefore, in the present framework, the classical Lebesgue inequality reads as
\begin{equation}
\label{LebesgueIneq}
\|\omega-P_{\mathcal T,w}\omega\|_0\leq (1+\L(\spaceV,\mathcal T,w))\inf_{\theta\in \spaceV}\|\omega-\theta\|_0\,.
\end{equation}

\subsection{Error estimates revisited}\label{sec:5.2}
A linear approximation scheme for elemants of $\testforms(E)$ can be set by picking an increasing sequence $\{\spaceVr\}_{r\in \N}$ of subspaces of $\testforms(E)$ such that $\cup_{r\in \N}\spaceVr$ is dense in $\testforms(E)$, i.e. $\min_{\theta\in \spaceVr}\|\omega-\theta\|_0\to 0$ as $r\to +\infty$ for any $\omega\in \testforms(E)$, and a sequence $\{\mathcal T^{(r)}\}_{r\in \N}$ of finite subsets of $\currents(E)$ such that, for any $r\in \N$, $\mathcal T^{(r)}=\{T^{(1,r)},\dots,T^{(M(r),r)}\}$ is a determining set for $\spaceVr$. Indeed, in such a setting we can consider the sequence $\{P^{(r)}\}_{r\in \N}$ of projection operators induced by weighted discrete least squares fitting on $\mathcal T^{(r)}$, that is
\begin{equation}\label{Pr}
P^{(r)}:=P_{\mathcal T^{(r)},w^{(r)}},
\end{equation}
where $w^{(r)}=\{w^{(1,r)},\dots,w^{(M(r),r)}\}$ is any positive finite sequence. 

The convergence analysis of this approximation scheme can be carried starting from the Lebesgue Inequality \eqref{LebesgueIneq}, to prove an upper bound for $\L(\mathcal T^{(r)},\spaceVr,w^{(r)})$ being the first key step in this direction.
The next proposition provides such a bound in the case of $\{\mathcal T^{(r)}\}_{r\in \N}$ being a $\{\spaceVr\}_{r\in \N}$-admissible integral $k$-mesh, while a sharper result in a similar direction is presented in Proposition \ref{corrollaryfitting} below. In order to state and prove these results it is convenient to introduce the following quantity
$$\kappa(y):=\|y\|_1\|(1/y_1,\dots, 1/y_{M})^t\|_\infty,\;\;\forall y\in]0,+\infty[^M\,.$$

\begin{proposition}\label{prop:errorestimatefittingmesh}
Let $E$ be a real body, let $\{\spaceVr\}_{r\in \N}$ be an increasing sequence of subspaces of $\testforms(E)$. Let $\{\mathcal T^{(r)}\}$ be a $\{\spaceVr\}_{r\in \N}$-admissible integral $k$-mesh of constant $C$. Let $w^{(s,r)}>0$ for any $s=1,\dots, M(r):=\Card \mathcal T^{(r)}$, $r\in \N$, and denote by $P^{(r)}$ the fitting operator defined in \eqref{Pr}. Then one has
\begin{equation*}
\L(\mathcal T^{(r)},\spaceVr,w^{(r)})\leq C\sqrt{\kappa(w^{(r)})}\ddim \spaceVr.
\end{equation*} 
Hence the following error estimate holds true for any $\omega\in \testforms(E)$:
\begin{equation}\label{standardfittingestimate}
\|\omega-P^{(r)}\omega\|_0\leq \left(1+ C\sqrt{\kappa(w^{(r)})}\ddim \spaceVr \right) \inf\{\|\omega-\theta\|_0,\,\theta\in \spaceVr\}\,.
\end{equation}
\end{proposition}
\begin{proof}
Recall that 
\begin{equation*}
\opnorm{P}\leq \L(\mathcal T^{(r)},\spaceVr,w^{(r)})= \sup_{T\in \averaging(E)}\sum_{s=1}^{M(r)}\left| \sum_{h=1}^NT^{(s,r)}(\eta_h)T(\eta_h)\right|w_i^{(r)}\,,
\end{equation*}
where $N:=\ddim \spaceVr$, see \cite[Th. 1]{BrPi25}. Using the sampling property of admissible integral $k$-meshes, we derive
\begin{align*}
\L(\mathcal T^{(r)},\spaceVr,w^{(r)})=&\leq C \sum_{s=1}^{M(r)}\max_{j\in\{1,\dots,M(r)\}}\left| T^{(j,r)}\left(\sum_{h=1}^NT^{(s,r)}(\eta_h)\eta_h\right)\right|w^{(s,r)}\\
\leq& C\sum_{s=1}^{M(r)}\max_{j\in\{1,\dots,M(r)\}}\left(\sum_{k=1}^N|T^{(j,r)}(\eta_k)|^2\right)^{1/2}\left( \sum_{h=1}^N|T^{(s,r)}(\eta_h)|^2\right)^{1/2}w^{(s,r)}\\
=& C\max_{j\in\{1,\dots,M(r)\}}\left(\sum_{k=1}^N|T^{(j,r)}(\eta_k)|^2\right)^{1/2}\sum_{s=1}^{M(r)}\left( \sum_{h=1}^N|T^{(s,r)}(\eta_h)|^2\right)^{1/2}w^{(s,r)}\\
\leq& C\max_{j\in\{1,\dots,M(r)\}}\left(\sum_{k=1}^N|T^{(j,r)}(\eta_k)|^2\right)^{1/2}\left(\sum_{s=1}^{M(r)} \sum_{h=1}^N|T^{(s,r)}(\eta_h)|^2w^{(s,r)}\right)^{1/2}\|w^{(r)}\|_1^{1/2}\\
\leq & \frac{C\|w^{(r)}\|_1^{1/2}}{\sqrt{\min_{s=1,\dots,M(r)}w^{(s,r)}}}\left(\sum_{j}^{M(r)}\sum_{k=1}^NT_j^2(\eta_k)w^{(j,r)}\right)^{1/2}\left( \sum_{h=1}^N\|\eta_h\|_{\mathcal T^{(r)},w^{(r)}}\right)^{1/2}\\
= & C\sqrt{\kappa(w^{(r)})}\left(\sum_{k=1}^N\sum_{j}^{M(r)}|T^{(j,r)}(\eta_k)|^2w_j^{(r)}\right)^{1/2}\sqrt{N}\\
=&C\sqrt{\kappa(w^{(r)})}N .
\end{align*}
The first claim is proved. The second claim follows from the first via \eqref{LebesgueIneq}. \myqed
\end{proof}
\begin{remark}[Fitting with equal weights]
In the case of equal weights, i.e. $w^{(r)}=const \cdot(1,1,\dots,1)$, one has $\kappa(w^{(r)})=M(r)$. Thus, \eqref{standardfittingestimate} simplifies to
$$\|\omega-P^{(r)}\omega\|_0\leq \left(1+C\sqrt{M(r)}\ddim\spaceVr)\right)\inf_{\theta\in \spaceVr}\|\omega-\theta\|_0.$$
Note that the value $M(r)$ is actually the global minimum of $\kappa$. However, we stress that, though equal weights strategy optimizes the upper bound \eqref{standardfittingestimate} of the Lebesgue constant, equal weights do not necessarily minimize the Lebesgue constant itself. A related discussion my be found in \cite{Piazzon19}.  
\end{remark}
 
Following the lines of \cite[Thm. 2]{CalviLevenberg08}, it is possible to obtain a much sharper result:
\begin{proposition}%(of \cite[Th. 2]{CalviLevenberg08})$\;\;$
\label{corrollaryfitting}
In the setting of Proposition \ref{prop:errorestimatefittingmesh}, one has
\begin{equation}\label{CL08estimate1}
\|P^{(r)}\omega\|_0\leq C\left(\|\omega\|_0+ \sqrt{\kappa(w^{(r)})}\right) \inf_{\theta\in \spaceVr}\|\omega-\theta\|_0,\quad \forall \omega\in \testforms(E),
\end{equation}
and
\begin{equation}\label{CL08estimate2}
\|\omega-P^{(r)}\omega\|_0\leq \left(1+C\left(1+ \sqrt{\kappa(w^{(r)})}\right)\right) \inf_{\theta\in \spaceVr}\|\omega-\theta\|_0,\quad \forall \omega\in \testforms(E)\,.
\end{equation}
\end{proposition}
\begin{proof}
Let us pick, for any $\epsilon>0$, $\theta_\epsilon\in \spaceVr$ such that 
$$\|\omega-\theta_\epsilon\|_0\leq \inf_{\theta\in \spaceVr}\|\omega-\theta\|_0+\epsilon\,.$$
Notice that, for any $r\in \N$ and $s\in\{1,\dots,M(r)\}$ one has
\begin{align*}
&|T^{(j,r)}(\omega)-T^{(j,r)}(P^{(r)}\omega)|\leq \frac 1{\sqrt{\min_{k=1,\dots,M(r)}w^{(k,r)}}}\|\omega-P^{(r)}\omega\|_{\mathcal T^{(r)},w^{(r)}}\\
\leq & \frac 1{\sqrt{\min_{k=1,\dots,M(r)}w^{(k,r)}}}\|\omega-\theta_\epsilon\|_{\mathcal T^{(r)},w^{(r)}}\\
\leq & \frac {\sqrt{\|w^{(r)}\|_1^{1/2}}}{\sqrt{\min_{k=1,\dots,M(r)}w^{(k,r)}}}\max_{l=1,\dots,M(r)}|T^{(l,r)}(\omega-\theta_\epsilon)|\\
\leq&\sqrt{\kappa(w^{(r)})}\|\omega-\theta_\epsilon\|_0\to \sqrt{\kappa(w^{(r)})}\inf_{\theta\in \spaceVr}\|\omega-\theta\|_0\;\text{ as }\epsilon\to 0.
\end{align*}
Taking the maximum over $j\in\{1,\dots, M(r)\}$ we obtain
\begin{equation}
\label{CL-1}
\|\omega-P^{(r)}\omega\|_{\mathcal T^{(r)}}\leq \sqrt{\kappa(w^{(r)})}\inf_{\theta\in \spaceVr}\|\omega-\theta\|_0\,.
\end{equation}
Therefore, using the sampling property of the integral admissible $k$-meshes, we can write
\begin{align*}
\|P^{(r)}\omega\|_0\leq & C\|P^{(r)}\omega\|_{\mathcal T^{(r)}}\leq C\left(\|P^{(r)}\omega-\omega\|_{\mathcal T^{(r)}}+\|\omega\|_{\mathcal T^{(r)}}\right)\\
\leq& C\left(\sqrt{\kappa(w^{(r)})}\inf_{\theta\in \spaceVr}\|\omega-\theta\|_0+\|\omega\|_{0}\right)\,.
\end{align*}
This proves \eqref{CL08estimate1}. In order to prove \eqref{CL08estimate2}, one may write
$$\|\omega-P^{(r)}\omega\|_0\leq \|\omega-\theta_\epsilon\|_0+\|P^{(r)}(\omega-\theta_\epsilon)\|_0,$$
where $\theta_\epsilon$ is as above, and then apply \eqref{CL08estimate1} to the second term. \myqed
\end{proof}
From \eqref{CL08estimate2} one deduces the advantage of using in the design of an approximation scheme admissible integral $k$-meshes with a mild growth rate of the cardinality of $\mathcal T^{(r)}$ as $r\to +\infty$. Considering for instance normalized equal weights, i.e. $w^{(s,r)}\equiv 1/M(r)$, and polynomial forms $\spaceVr=\forms{r}{k}(E)$, one indeed has
$$\|\omega-P^{(r)}\omega\|_0\leq \left(1+C\left(1+ \sqrt{ M(r)}\right)\right) d_r(\omega,E )\sim 2\widetilde C\sqrt{M(r)}d_r(\omega,E ),\quad \forall \omega\in \testforms(E)\,,$$
where $d_r(\omega,E)=\inf_{\theta\in \forms{r}{k}(E)}\|\omega-\theta\|_0$.
This last estimate, for a \emph{quasi optimal} admissible integral $k$-mesh (i.e., $M(r)\sim \ddim \forms{r}{k}\log r  $) as the one constructed on the simplex by the Algorithm \ref{alg:meshthesimplex} of Section \ref{sec:4}, implies 
$$\|\omega-P^{(r)}\omega\|_0\sim 2 \widetilde C\sqrt{\ddim \forms{r}{k}\log r\,}d_r(\omega,E )\sim 2 \widetilde C\sqrt{\binom{n}{k}\binom{n+r}{n}  \log(r)} d_r(\omega,E),$$
while, for the case of an \emph{optimal} admissible integral $k$-mesh  (i.e., $M(r)\sim \ddim \forms{r}{k} $) as that of Example \ref{ex:baransquare}, the estimate \eqref{CL08estimate2} reads
\begin{equation}\label{erroronoptimalmesh}
\|\omega-P^{(r)}\omega\|_0\sim 2 \widetilde C\sqrt{\ddim \forms{r}{k}\,}d_r(\omega,E )\sim 2 \widetilde C\sqrt{\binom{n}{k}\binom{n+r}{n}  } d_r(\omega,E)\,,
\end{equation}
see Remark \ref{rmk:comparecardinality}. 

We remark that in such cases the convergence of $P^{(r)}\omega\to \omega$ with respect to the zero norm easily follows under additional mild assumptions on $\omega$ and $E$. For instance, if the multivariate Jackson Inequlality \cite{Pl09} holds on $E$, and $\omega$ has $\mathscr C^k$ coefficients for some $k>n/2$, then \eqref{erroronoptimalmesh} implies that $\|\omega-P^{(r)}\omega\|_0\sim r^{n/2-k}\to 0.$

%%%%%%%

  % fitting
\section{\textquotedblleft Good\textquotedblright\ unisolvent arrays of currents extracted from admissible integral $k$-meshes}\label{sec:interp}

A unisolvent set $\mathcal T$ for the linear space $\spaceV\subset \testforms(E)$ is a determining set $\mathcal T\subset \currents(E)$ that achieves the minimal cardinality $\Card \mathcal T=N:=\ddim \spaceV$. In such a case the operator $P_{\mathcal T,w}$ introduced in \eqref{eq:samplingProjector} becomes independent by $w$ and reduces to the interpolation operator $\Pi_{\mathcal T}$, i.e., $\Pi_{\mathcal T}\omega $ is the unique element of $\spaceV$ such that $T_i(\Pi_{\mathcal T})=T_i(\omega)$, $i=1,2,\dots,N$.

 Indeed, in such a setting there exists a Lagrange basis $\omega_1,\dots,\omega_N$ of $\spaceV$, i.e. such that $T_i(\omega_j)=\delta_{i,j}$, and one has
$$\Pi_\mathcal T(\omega):=\sum_{i=1}^N T_i(\omega)\omega_i\,.$$
 
 Note that, in such interpolatory case , the right hand side of Eq. \eqref{Lebesgueconstant} simplifies to 
 $$\L(\spaceV,\mathcal T)=\sup_{T\in \averaging(E)}\sum_{i=1}^N|T(\omega_i)|,$$
 where $\omega_i$'s form the aforementioned Lagrange basis, and the Lebesgue Inequality \eqref{LebesgueIneq} holds with this definition of $\L$.

\subsection{Extremal sets of currents of Fekete and Leja type}\label{sec:5.1}
Let us assume that $\{\spaceVr\}_{r\in \N}$ is an increasing sequence of linear subspaces of $\testforms(E)$ with $\ddim \spaceVr=N(r)$ and, for any $r\in \N$, $\mathcal T^{(r)}:=\{T^{(s,r)}\}_{s=1}^{N(r)}\subset \currents(E)$ is a unisolvent set for $\spaceVr$. Then we can define the sequence of linear projection operators $\Pi^{(r)}:\testforms \rightarrow \spaceVr$ by setting
$$\Pi^{(r)}:=\Pi_{\mathcal T^{(r)}}\,.$$

We already pointed out that, in order to obtain the convergence $\|\omega-\Pi^{(r)}\omega\|_0\to 0$ for $\omega$ ranging in a reasonably large class of differential forms, it is mandatory to hold the growth of $\L(\spaceVr,\mathcal T^{(r)})$ compared to the decay of $\inf_{\theta\in \spaceVr}\|\omega-\theta\_0$. Arrays of points with sub-exponentially increasing Lebesgue constant are customarily termed \emph{good interpolation points} in the context of polynomial interpolation.

Already in the case $k=0$ and $n=1$ (i.e., nodal interpolation of univariate functions) and for the polynomial case (i.e., $\spaceV=\mathscr P_r$), the minimization of the Lebesgue constant among all possible set of interpolation nodes on a given compact set $E$ (the \emph{Lebesgue problem}) becomes unmanageable already for mild values of $r$. The quest for \emph{computable} sub-optimal solutions is thus of major interest in approximation theory. 

Fekete points are the most studied sub-optimal solution of the Lebesgue problem. A set of nodes $X:=\{x_1,x_2,\dots,x_N\}\subset E^N$ with $N:=\ddim \mathscr P_r$, is termed a set of \emph{Fekete points} for $E$ if it maximizes the determinant of the Vandermonde matrix $\vdm$:%, denoting by $\vdm$ the ,
\begin{equation}\label{eq:feketedef}
|\det \vdm(X,\mathcal B_r)|=\max_{Y\in E^N}|\det \vdm(Y,\mathcal B_r)|,\;\;\text{with }\vdm(Y,\mathcal B_r)_{i,j}=b_j(y_i),
\end{equation}
for one, and thus for all, basis $\mathcal B_r=\{b_1,\dots,b_N\}$ of $\mathscr P_r.$ Fekete problems in the segmental framework have also been studied \cite{BEFekete}.

Another proposed sub-optimal solution to the Lebesgue problem are Leja sequences. \emph{Leja sequences} are defined by means of the greedy version of the maximization procedure appearing in \eqref{eq:feketedef}. Precisely one first picks $x_1\in E$, then iteratively sets, for $ k\in \N$, 
\begin{equation}\label{eq:lejaseqdef}
x_{k+1}\in {\argmax}_{x\in E}|\det \vdm(\{x_1,\dots,x_k,x\},\{b_1,\dots,b_{k+1}\})|\,.
\end{equation} 
The first interest on Fekete points is easy to see: if $\{x_1,x_2,\dots,x_N\}$ are Fekete points, then we can write
$$\L(X,\mathscr P_r)=\max_{x\in E}\sum_{i=1}^N |\ell_i(x)|=\max_{x\in E}\sum_{i=1}^N \frac{|\det \vdm(X^{(i)}(x),\mathcal B_r)|}{|\det \vdm(X,\mathcal B_r)|}\leq \sum_{i=1}^N 1=N,$$
where we denoted by $\ell_i$ the $i$-th Lagrange polynomial and by $X^{(i)}(x)$ the set $X$ with the $i$-th element replaced by $x$. Hence Fekete points have a polynomially increasing Lebesgue constant.

%There is also a perhaps even more important, reason for which Fekete points have been investigated. Indeed, 
The relationship between Lebesgue problem and the maximization of the Vandermonde determinant also bridges approximation theory with logarithmic potential theory when $n=1$ \cite{SaffTotik97,Ransford95} and pluripotential theory when $n>1$ \cite{Norm12}. % We refer to for an exhaustive treatment of the subject in one dimension and to  for a survey on theese deep connections when $n>1$. 
This link is offered by the asymptotic of the optimal Vandermonde determinant as $r\to +\infty$, which is the \emph{transfinite diameter} of the set $E$:%It turns out that the optimal Vandermonde determinant has an extremely relevant asymptotics as $r\to +\infty$, the \emph{transfinite diameter} of the set $E$:
\begin{equation}\label{transfinitediameter}
	\delta(E)=\lim_r\delta^{(r)} (E) := \left[ \max_{\boldsymbol x\in E^N} \Big| \vdm(\{x_1,\dots,x_N\},\mathcal B_r^{\text{mon}}) \Big| \right]^{1/(l_r)},
\end{equation}
where $l_r:=\sum_{j=1}^r[j(\ddim \mathscr P_j-\ddim \mathscr P_{j-1})],$ and $\mathcal B_r^{\text{mon}}$ is the monomial basis introduced in Section \ref{sect:polydf}. The existence of the limit in the definition of $\delta(E)$ is rather straightforward if $n=1$, while was proved for $ n > 1 $ by Zaharjuta \cite{Za75} with more complicated technologies. %, whereas, for $n>1$, it has been only conjectured for a long time, being finally proven by Zaharjuta \cite{Za75}. 
%The transfinite diameter (and its \emph{weighted} generalization) plays a pivotal role in (pluri-)potential theory. 
In one complex variable, $\delta(E)$ is the \emph{logarithmic capacity} of $E$, the quantity that distinguishes (from the potential-theoretic point of view) relevant sets from negligible ones. %We can also mention that a
Arrays of interpolation nodes leading to the transfinite diameter of $E$ %(i.e. such that $\lim_r| \vdm(\{x_1,\dots,x_N\},\mathcal B_r)|^{1/(l_r)}=\delta(E)\neq 0$) or, equivalently, 
are termed \emph{asymptotically Fekete} and necessarily tend, in the weak$^*$ topology of measures, to $\mu_E$, the \emph{equilibrium measure} of the compact set $E$.%, that is the probability measure on $E$ given by the normalized distributional Laplacian of the Green fnction of $\C\setminus E$ with logarithmic pole at $\infty$.  

When $n>1$ the situation becomes more involved. Nevertheless, all the above mentioned relations among Fekete points, transfinite diameter, and suitably generalized versions of equilibrium measure, Green function, and capacity, have been extended replacing the potential-theoretic point of view by %their counterpart in 
the framework of pluripotential theory. In particular, the sequence of uniform probability measures supported at Fekete points converges to $\mu_E$, the pluripotential equilibrium measure of $E$ \cite{BeBo10,BeBoNy11}.

A further step is required when the geometry is enriched and differential forms are considered. In \cite{BruPia24} authors introduced the (weighted) transfinite diameter of $E$ with respect to a real vector space $U$. For that, a choice of an orthonormal basis of $U$ is required. This definition can be specialized to the case of $U=\Lambda^k$ via the orthonormality introduced in Eq. \eqref{eq:samplingProjector}, leading to the following definition:
\begin{align}
\delta(E,\Lambda^k):=&\lim_r\delta^{(r)}(E,\Lambda^k)\notag\\
\delta^{(r)}(E,\Lambda^k):=&\max_{\mathcal T\in [\currents(E)]^{N(r)}, \Mass(T)=1}|\det\vdm(\mathcal T,\mathcal B_r^{\text{mon},k})|^{1/(\binom{n}{k}l_r)},\label{eq:transfinitediameterlambdak}
\end{align}
where $\Mass$ denotes the mass-norm defined in Subsection \ref{sec:2.2}, %with respect to the mass $M(T):=\sup\{|T(\omega)|: \omega\in \testforms(E),\,\|\omega\|_0\leq 1\}$, 
and $\mathcal B_r^{\text{mon},k}$ is the monomial basis introduced in Subsection \ref{sect:polydf}. The existence of the limit in equation \eqref{eq:transfinitediameterlambdak} is proved in \cite{BruPia24} in a weighted setting. Further, the weighted transfinite diameter with respect to $ \Lambda^k $ can be expressed as a geometric mean of standard weighted transfinite diameters, simplifying in the unweighted case to
\begin{equation*}\label{eq:equalityofdiameters}
\delta(E,\Lambda^k)=\delta(E).
\end{equation*}

With this at hand, we can extend the definition of Fekete points and asymptotically Fekete arrays.

\begin{definition}[Fekete currents] \label{def:Fekete}
Let $\mathcal{F}^{(r)}:=\{\mathcal{F}^{(1,r)},\dots,\mathcal{F}^{(N(r),r)}\}\subset \currents(E)$, $\Mass(\mathcal F^{(s,r)})=1$ for $s=1,\dots,N(r)=\ddim \forms{r}{k}$. Then $\{\mathcal{F}^{(1,r)},\dots,\mathcal{F}^{(N(r),r)}\}$ are termed \emph{Fekete currents} if
\begin{equation*}
|\det\vdm(\mathcal{F}^{(r)},\mathcal B_r^{\text{mom},k})|=[\delta^{(r)}(E,\Lambda^k)]^{\binom{n}{k}l_r  }.
\end{equation*} 
\end{definition}
The sequence $\{\widetilde {\mathcal{F}}^{(r)}\}_{r\in \N}$ is termed \emph{asymptotically Fekete} if 
\begin{equation} \label{eq:asymptoticallyfekete}
\lim_r|\det\vdm(\widetilde {\mathcal{F}}^{(r)},\mathcal B_r^{\text{mom},k})|^{1/\left(\binom{n}{k}l_r  \right)}=\delta(E).
\end{equation}
Since the asymptotics discussed above also hold for asymptotically Fekete arrays, see \cite{BlChLe92}, %such as, e.g., Leja sequences defined in \eqref{eq:lejaseqdef} above, see , Also 
Leja sequences defined in \eqref{eq:lejaseqdef} can be similarly extended to the currents based setting. As for the case $k=0$, this definition depends on the ordering of the polynomial basis. In particular we need the following notion:
\begin{definition}[Lower triangular basis]\label{def:admissiblebasis}
Let $\mathcal Q=\{q_1,q_2,\dots\}$ be a basis for $\forms{}{k}$. We say that $\mathcal Q$ is lower triangular if for any $N\in \N$ there exist a lower triangular matrix $L^{(N)}$ such that 
\begin{equation*}\label{eq:admisssiblebasis}
q_i=\sum_{j=1}^{N} L^{(N)}_{i,j}b_j=\sum_{j=1}^{i} L^{(N)}_{i,j}b_j\,,
\end{equation*} 
where $\kmonomials{}{k}=\{b_1,b_2,\dots\}.$
\end{definition}
\begin{definition}[Leja sequences of currents] \label{def:Leja}
Let $\mathcal Q=\{q_1,q_2,\dots\}$ be a lower triangular basis of $\forms{}{k}=\cup_{r\in \N}\forms{r}{k}$ in the sense of Definition \ref{def:admissiblebasis}, and, for any $N\in \N$, denote by $\mathscr V_N$ the linear space $\Span\{q_1,\dots,q_N\}$. Let
$$L_1\in \argmax\{|T(q_1)|,\,T\in \currents(E), \Mass(T)=1\},$$
and let, for any $i\in \N$,
\begin{equation*}
L_{i+1}\in \argmax \{|\det\vdm(\{L_1,\dots,L_i,L\},\mathscr V_{i+1})|, \,L\in \currents(E),\;\Mass(L)=1\}\,.
\end{equation*}
The sequence $\{L_i\}$ is a Leja sequence for $\forms{}{k}$ relative to $\mathcal Q$.
\end{definition}
%We remark that the asymptotic behaviour of (asymptotic-)Fekete currents has been investigated in \cite{BruPia24}, 
The problem of finding Fekete currents has some simplifications with respect to the Lebesgue problem. Nevertheless, it is still completely unfeasible from a computational point of view. Polynomial admissible mesh have been successifully employed to obtain a further simplification in the case of nodal interpolation of functions  \cite{BosCalviLevenbergSommarivaVianello11}.
Indeed,  one may replace, in Definition \ref{def:Fekete} and Definition \ref{def:Leja}, the domain of the maximization $\{T\in\currents(E): \Mass(T)=1\}$ with any admissible integral $k$-mesh. In such a case, we will refer to the corresponding currents as \emph{Fekete and Leja sequences extracted from the mesh}. Also in the context of the present work this approach is particularly profitable, as shown by the following theorem. 

\begin{theorem}\label{th:lebesgueforAFP}
Let $E\subset \R^n$ be a compact non-pluripolar set (see for instance \cite[\S 2.9]{Kl91}), $k\in\{1,\dots,n\}$ and let $\{\mathcal T^{(r)}\}$ be an admissible integral $k$-mesh for $E$ of constant $C$. Any sequence $\{\widetilde {\mathcal F}^{(r)}\}$ of Fekete or Leja currents extracted from the mesh $\{\mathcal T^{(r)}\}$ is asymptotically Fekete. In particular
\begin{equation} \label{eq:LebesgueFekete}
\L (\widetilde{\mathcal F}^{(r)},\forms{r}{k})\leq C \ddim \forms{r}{k}. 
\end{equation}
Hence $\left(\L (\widetilde {\mathcal F}^{(r)},\forms{r}{k})\right)^{1/r}\leq  \left(C\binom{n}{k}\binom{n+r}{r}\right)^{1/r}\sim \left(\frac{C}{(n-k)!k!}r^n\right)^{1/r}\to 1$ as $r\to +\infty$.
\end{theorem}
\begin{proof}
    Let $ \{ \mathcal{F}^{(r)} \} $ be any \emph{true} Fekete sequence of currents for $ E $. Let $ \{ \widetilde {\mathcal F}^{(r)} \} $ be the Fekete currents extracted from the mesh $ \mathcal{T}^{(r)} $ and $ \{ \widetilde L_i \} $ be the Leja sequence extracted from the same mesh. Using Definition \ref{def:Fekete} and Definition \ref{def:Leja} and the sampling property \eqref{constantofmeshHk}, one can write
    \begin{align*} |\det\vdm({\mathcal{F}}^{(r)},\mathcal B_r^{\text{mom},k})| & \geq |\det\vdm(\widetilde {\mathcal{F}}^{(r)},\mathcal B_r^{\text{mom},k})| \geq |\det\vdm(\{\widetilde L_1, \ldots, \widetilde L_{N(r)}\},\mathcal B_r^{\text{mom},k})| \\ 
    & \geq  \frac{1}{C^{N(r)}} |\det\vdm({\mathcal{F}}^{(r)},\mathcal B_r^{\text{mom},k})| .
    \end{align*}
    Taking the $ \binom{n}{k} l_r $-root of each term, one sees that the first and the last terms satisfy \eqref{eq:asymptoticallyfekete}, and so do all the intermediate terms.

    To prove the second part of the statement, consider the Lagrange basis $ \{ \omega_1, \ldots, \omega_{N(r)} \} $ associated with the set of currents $ \widetilde{\mathcal F}^{(r)} $. It follows from the definition of Fekete currents extracted from $ \mathcal{T}^{(r)}$ that
    $$ \Vert \omega_i \Vert_{\mathcal{T}^{(r)}} = \max_s \frac{\left| \det\vdm(\{ \widetilde F^{(1,r)}, \ldots,  \widetilde F^{(i-1,r)} , T^{(s,r)},  \widetilde F^{(i+1,r)}, \ldots,  \widetilde F^{(N,r)} \},\mathcal B_r^{\text{mom},k})\right| }{|\det\vdm({\widetilde {\mathcal{F}}}^{(r)},\mathcal B_r^{\text{mom},k})|} \leq 1 .$$
    Using the sampling property \eqref{constantofmeshHk} one gets $ \Vert \omega_i \Vert_0 \leq C \Vert \omega_i \Vert_{\mathcal{T}^{(r)}} $, whence taking the sum for $ i = 1, \ldots, N(r) $, the claim follows. \myqed
\end{proof}
\begin{remark}
To the best of the authors' knowledge, Theorem \ref{th:lebesgueforAFP} provides the first example of interpolation schemes based on integration of differential forms having \emph{Lebesgue constant of polynomial growth}. In particular, the classical Lebesgue estimate for interpolation in the setting of Theorem \ref{th:lebesgueforAFP} reads 
\begin{equation*}
\|\omega-\Pi^{(r)}\omega\|_0\leq \left(1+C \ddim \forms{r}{k}\right) d_r(\omega,E),
\end{equation*}
where $d_r(\omega,E):=\in\{\|\omega-\theta\|_0,\,\theta\in \forms{r}{k}\}$ is the error of best approximation. 
\end{remark}

The search of Fekete (or Leja) currents extracted from an integral admissible $k$-mesh instead of true Fekete (or Leja) currents suggested by Theorem \ref{th:lebesgueforAFP} moves the problem from the context of continuous optimization to the discrete one. However, the problem on the discrete level is still NP-hard \cite{Civril09}. For this reason, heuristic or stochastic approaches are generally pursued \cite{BosDeMarchiSommarivaVianello10,BosCalviLevenbergSommarivaVianello11}.  All the new objects introduced in the present paper have been suitably defined in order to extend to differential forms the constructions of \emph{approximate Fekete points} (AFP algorithm) or of \emph{discrete Leja points} (DLP algorithm) proposed in \cite{BosDeMarchiSommarivaVianello10}. These algorithms rest upon standard numerical linear algebra (QR and LU factorizations, respectively) whose efficiency and stability has deeply investigated. In particular, the AFP algorithm (paired with the so called twice-is-enough principle \cite{BoDeSoVi10}) experimentally demonstrates a great robustness, being capable of selecting sets of interpolation nodes with small Lebesgue constant even starting from Vandermonde matrices having a condition number comparable to the inverse of machine precision. 

An example of Fekete currents extracted by the AFP algorithm is depicted in Figure \ref{fig:AFPsimplex}.
\begin{remark}
Consistetly with the nodal case, the AFP and DLP algorithms provably extract collections of asymptotically Fekete currents for differential forms. This is easily achieved reproducing the proof of \cite[Thm. 1]{BosCalviLevenbergSommarivaVianello11}. In contrast, even for $ k = 0 $, it has not been proven that the algorithms output sequences of currents having sub-exponentially growing Lebesgue constant, i.e., $ \limsup_{r\to \infty} \left(\L (\widetilde {\mathcal F}^{(r)},\forms{r}{k})\right)^{1/r} \leq 1 $. Nevertheless, numerical experiment carried in the nodal framework suggest that such Lebesgue constants exhibit polynomial growth \cite{BoDeSoVi11b}.
\end{remark}

\begin{figure}[H]
\caption{Two examples of approximate Fekete currents extracted by the AFP algorithm from an admissible integral $1$-mesh on the unit $2$-simplex. \emph{Left: } $r=5$, $C=3$, $\L\approx 3.8$. \emph{Right: } $r=7$, $C=3$, $\L\approx 7$.}\label{fig:AFPsimplex}
\begin{center}
\begin{tabular}{cc}
\includegraphics[width=0.49\textwidth]{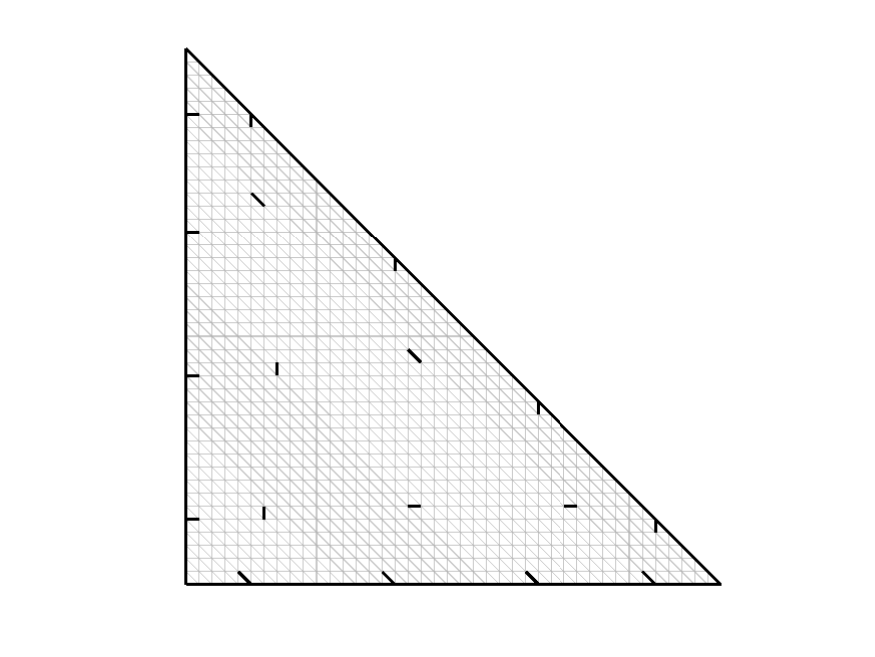}&
\includegraphics[width=0.49\textwidth]{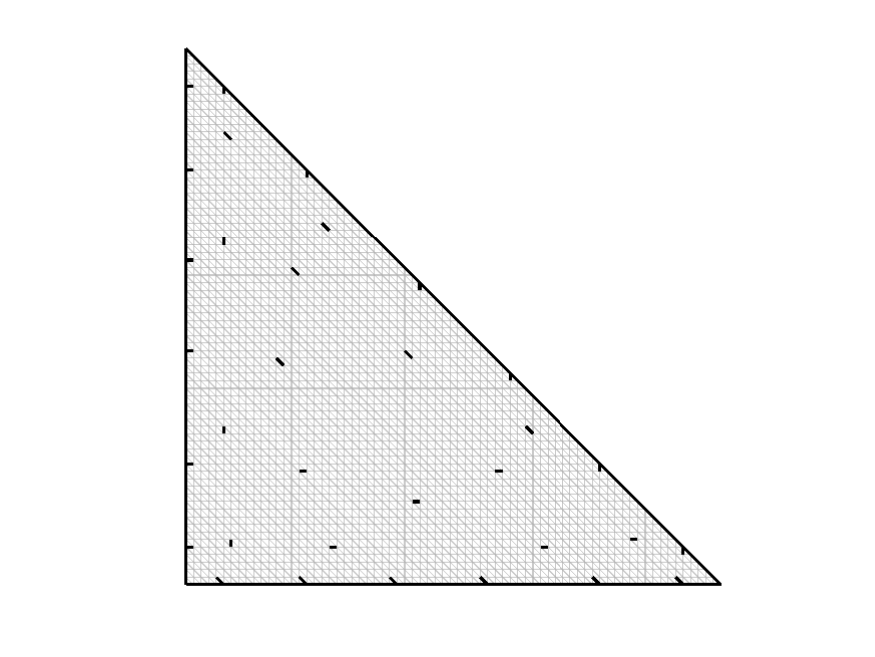}
\end{tabular}
\end{center}
\end{figure}

\bibliographystyle{plain}
\bibliography{bibliography}

\end{document}